\documentclass[AMS]{article} 

\usepackage{xcolor}
\usepackage{geometry} 
\usepackage{graphicx} 

\usepackage{booktabs} 
\usepackage{array} 
\usepackage{paralist} 
\usepackage{verbatim} 
\usepackage{subfig} 

\usepackage{fancyhdr} 
\usepackage{amssymb}
\usepackage{amsmath}
\usepackage{epstopdf}
\usepackage[T1]{fontenc}
\usepackage{t1enc}
\usepackage[utf8]{inputenc}

\usepackage{sectsty}
\allsectionsfont{\sffamily\mdseries\upshape} 

\usepackage{hyperref}

\usepackage[nottoc,notlof,notlot]{tocbibind} 
\usepackage[titles,subfigure]{tocloft} 

\newtheorem{theorem}{Theorem}[section]
\newtheorem{lemma}[theorem]{Lemma}
\newtheorem{obs}[theorem]{Observation}
\newtheorem{fact}[theorem]{Fact}
\newtheorem{cor}[theorem]{Corollary}

\newtheorem{claim}[theorem]{Claim}
\newtheorem{prop}[theorem]{Proposition}
\newtheorem{defi}[theorem]{Definition}

\newcommand{\ep}{\varepsilon}

\newcommand{\cL}{\mathcal{L}}

\newcommand{\cC}{\mathcal{C}}

\newcommand{\cT}{\mathcal{T}}

\newcommand{\K}{\mathcal K}

\def\qedf{\hfill $\Box$}

\title{A stability theorem for embedding
bounded degree spanning trees} 
\author{ B\'ela Csaba\thanks{The research of the author was supported
by the Ministry of Innovation and Technology of Hungary from the National Research, Development and
Innovation Fund, project no. TKP2021-NVA-09. E-mail: bcsaba@math.u-szeged.hu. ORCID: https://orcid.org/0000-0002-6696-3219}\\ Bolyai Institute, University of Szeged, Hungary. }
\date{} 

\begin{document}
\maketitle

\begin{abstract}
We prove that if an $n$-vertex graph $G$ is non-extremal and $T$ is a bounded degree tree on $n$
vertices, then $T\subset G$ even when the minimum degree of $G$ is less than $n/2$ by a linear term.
We avoid the use of the Regularity lemma, instead we apply a vertex decomposition theorem by the author,
which does not require a tower-type lower bound for $n.$ 
\end{abstract}

\medskip

\quad {\bf keywords: } tree embedding, decomposition of graphs, regularity

\medskip

\section{Introduction}

We consider only simple graphs in this paper.
In 1978 Bollob\'as~\cite{BB} conjectured that if $G$ is a graph on $n$ vertices, $n$ is large enough, and 
$\delta(G) \ge (1/2+\ep)n$ for some $\ep >0,$ and $T$ is a bounded
degree tree on $n$ vertices, then $T \subset G.$ The problem was solved in the affirmative by Koml\'os, S\'ark\"ozy and Szemer\'edi~\cite{KSSz1}
for large graphs. Few years later they strengthened their result (see~\cite{KSSz2}), and proved, that $\Delta(T)$ need not be bounded: there exists a constant $c$ such that 
$T \subset G$ if $\Delta(T) \le cn/\log n,$
$\delta(G) \ge (1/2+\ep)n$ and $n$ is large. The Regularity lemma played an essential role in both proofs of the Bollob\'as conjecture.
In~\cite{regifa} the author, Levitt, Nagy-Gy\"orgy and Szemer\'edi proved the theorem below, without using the Regularity lemma.

\begin{theorem}\label{rfa}
Assume that $T$ is a bounded degree tree on $n$ vertices with $\Delta(T)=D.$
Let $G$ be a graph on $n$ vertices. Assume further that $\delta(G) \ge n/2 +c_D\log{n},$ where $c_D$ is a constant depending
only on $D.$
Then there exists a number $n_0$ such that $T \subset G$ for $n \ge n_0$. Furthermore, the bound on $\delta(G)$ is tight:  there exists a graph $G$ with
$\delta(G) > n/2+\log {n}/17$ such that the complete ternary tree on $n$ vertices is not a subgraph of $G$ if $n$ is sufficiently large.
\end{theorem}

That is, for embedding a bounded degree tree on $n$ vertices, it is sufficient if the minimum degree of the host graph is $n/2+O(\log n),$ 
moreover, an additive term of $\Omega({\log n})$ above $n/2$ is necessary. 

Let us mention a closely related theorem by Kathapurkar and Montgomery~\cite{KM}. They proved, without the Regularity lemma, that if the minimum semidegree in a $n$-vertex directed graph $D$ is at least $(1/2+\ep)n,$ then $D$ contains every spanning
oriented tree with maximum degree $o(n/\log n).$ This theorem implies the analogous result in~\cite{KSSz2}
for spanning trees in undirected graphs.

The main result of the present paper is a stability theorem for embedding bounded degree trees, strengthening 
Theorem~\ref{rfa} for so called non-extremal host graphs. First we need a definition for non-extremality.
Let $0<\gamma < 1/4$ be a number. A graph $G=(V, E)$ on $n$ vertices is called {\it $\gamma$-non-extremal}, if 
for every, not necessarily disjoint subsets $A, B \subset V,$ $|A|=|B|=\lfloor n/2\rfloor$ we have $e(A, B)> \gamma n^2,$ otherwise we
say that $G$ is {\it $\gamma$-extremal.} Note, that $\gamma\le 1/4$ for every graph, and the only $1/4$-non-extremal graph is the complete graph.

It is well-known (see eg.~in~\cite{regifa}) that if $G$ has minimum degree at least $n/2$ and is extremal, then, roughly speaking, either it is close to a balanced complete bipartite graph, or to the union of two vertex disjoint complete
graphs on $n/2$ vertices each. Here closeness is measured in edit distance. 

Let us remark, that due to the minium degree conditions, the Bollob\'as conjecture and the two papers~\cite{KSSz1, KSSz2}
by Koml\'os, S\'ark\"ozy and Szemer\'edi are on embedding spanning trees into $\ep$-non-extremal host graphs, similarly
to the paper by Kathapurkar and Montgomery~\cite{KM}.
It turns out that if $G$ is non-extremal, then the minimum degree requirement can be relaxed significantly. 

\begin{theorem}\label{fa}
Let $D\ge 2$ be an integer and $\gamma, \nu$ real numbers with $0<\nu \ll \gamma \ll 1.$ Then there exists a threshold number $n_0=n_0(D, \gamma, \nu)$ such that the following holds. If $n\ge n_0,$ $G$ is a $\gamma$-non-extremal graph on $n$ vertices having minimum degree $\delta(G)\ge (1/2-\nu)n,$ and $T$ is an $n$-vertex tree with $\Delta(T)\le D,$ then 
$T\subset G.$
\end{theorem}

Observe that, according to the above theorem, the minimum degree of $G$ can be strictly less than $n/2,$
if $G$ is non-extremal. 
Careful examination of the proof shows that with parameters $\gamma\le 1/100$ and $\nu\le \gamma/50,$ 
the theorem holds. However, it is easier to follow the computations if throughout the proof we apply the more flexible hierarchy $0< \nu \ll \gamma \ll 1$ in the statement of the theorem.

The proof of Theorem~\ref{rfa} has three cases: the case of a non-extremal $G,$ the almost complete bipartite case and 
the union of two almost 
complete subgraphs case. The minimum degree requirement of Theorem~\ref{rfa} is tight for the extremal
cases, as is proved in~\cite{regifa}, while Theorem~\ref{fa} shows that it can be substantially relaxed for non-extremal host graphs.
Hence, the stability version we prove in this paper together with the proofs 
of the extremal cases in~\cite{regifa} not only gives a new proof for this tree embedding problem, it is
also a stronger result. 

In the proof of Theorem~\ref{fa} we use several tools, ideas that were developed for use with the Regularity lemma of Szemer\'edi~\cite{Szemeredi}, but we replace the Regularity lemma with an alternative graph decomposition result by the author~\cite{CsBundle}. 
Roughly speaking,
we can use the advanced, well developed techniques of the Regularity Method -- using only a ``super matching'' containing super-regular pairs, which we find without the Regularity lemma. Due to this fact the new proof applies for graphs of ``practical'' size, not just for astronomically large ones.

The paper is organized as follows. In the second section we review the basic definitions and tools we need. Beginning with the third section we focus on the proof of the tree embedding theorem. The third section includes the preprocessing of the host graph $G,$ the fourth section contains the preprocessing of the tree to be embedded. A few further tools
are introduced in the fifth section. Finally, in the sixth section we put everything together, and prove the tree embedding theorem.

\section{Notation, definitions, main tools}

Given a graph $G=(V, E)$ we use the notation $v(G)=|V|$ and $e(G)=|E|.$ Given a set 
$X\subset V,$ $G[X]=(X, E_X),$ where $E_X=\{uv: uv\in E, \ u, v\in X\}.$ For 
disjoint subsets $X, Y\subset V,$ we let 
$G[X, Y]$ denote the bipartite subgraph of 
$G$ with parts $X$ and $Y$ that contains all the edges of $G$ with one endpoint in $X$ and the other endpoint in $Y.$ For every vertex $v\in V,$ the {\it neighborhood} of $v$ is denoted by 
$N_G(v),$ and the {\it degree} of $v$ is denoted by $\deg_G(v)=|N_G(v)|.$ Given a set $S\subset V,$ we let $N_G(v, S)=N_G(v)\cap S$ and 
$\deg_G(v, S)=|N_G(v, S)|,$ the subscripts may be omitted. For an $H\subseteq G$ and $v\in V(G),$ the number of neighbors of $v$ in $H$
is sometimes denoted by $deg_G(v, H).$ If we write $G=(A, B; E),$ this means that $G$ is bipartite with parts $A$ and $B,$ and edge set $E.$ If it is clear from the context, that a graph in question is bipartite, we may only write out
the vertex parts, and omit the letter ``$E$''. We call a bipartite graph {\it balanced}, if the two parts have the same cardinality.

The {\it density} of $G$ is defined to be $d_G=e(G)\cdot {\binom{v(G)}{2}}^{-1}.$ The {\it bipartite density} of bipartite subgraphs of $G$ with parts $A$ and $B$ 
is $d_G(A, B)=\frac{e(G[A, B])}{|A|\cdot |B|}.$ Sometimes the subscript may be omitted when there is no confusion. Similarly, when a graph in question is bipartite, density will mean bipartite density.

Given numbers $x, y$ we say that $z=x\pm y,$ if $x-y \le z \le x+y.$ If $n\ge 1$ is an integer, then we let 
$[n]=\{1, \dots, n\}.$ 
For numbers $0< \alpha, \beta <1$ the notation ``$\alpha \ll \beta$'' means that $\alpha$ is sufficiently smaller than $\beta.$ We remark, that whenever this notation is used in the paper, the relation of $\alpha$ and $\beta$ can be explicitly calculated, and will always
mean that $\alpha \le \beta^c$ for some positive integer constant $c.$ Still, using ``$\ll$'' enables us to concentrate on the essential parts
of the proofs.

\subsection{Regular pairs}

While we avoid using the Regularity lemma, the notion of regularity plays an essential role in the paper. Below 
is a brief review of the basics in the area.

\medskip

\begin{defi} 
Let $0<\ep, \delta <1$ be real numbers.
We say that a bipartite graph $H=(A, B)$ is an $\varepsilon$-{\it regular pair}, if for every $X\subseteq A,$ $Y\subseteq B$ 
with $|X|\ge \varepsilon |A|$ and $|Y|\ge \varepsilon |B|$ we have 
$$|d_H(A, B) - d_H(X, Y)|\le \ep.$$ 

We call $H$ an $(\varepsilon, \delta)$-{\it super-regular pair}, if in addition every $v\in A$ has at least $\delta |B|$ neighbors and every $u\in B$ has at least 
$\delta |A|$ neighbors.
\end{defi}


The following well known fact below will prove to be useful, the proof is omitted.

\begin{fact}\label{reg}
Assume that $H=(A, B)$ is an $\ep$-regular pair with density $d_H.$ Let $A'=\{x\in A: deg(x)<(d_H-\ep)|B|\}$ and 
$A''=\{x\in A: deg(x)>(d_H+\ep)|B|\}.$ Similarly, let $B'=\{x\in B: deg(x)<(d_H-\ep)|A|\}$ and 
$B''=\{x\in B: deg(x)>(d_H+\ep)|A|\}.$ Then $|A'|, |A''|\le \ep |A|$ and $|B'|, |B''|\le \ep |B|.$  
\end{fact}

We will use the so called Slicing lemma~\cite{KS} at various points in the paper.

\begin{fact}\label{slicing}
Assume that $H=(A, B)$ is an $\ep$-regular pair with density $d_H,$ and for some $\alpha >\ep$ let $A'\subset A,$
$|A'|\ge \alpha |A|$ and  $B'\subset B,$ $|B'|\ge \alpha |B|.$ Then $(A', B')$ is an $\ep'$-regular pair with 
$\ep'=\max \{\ep/\alpha, 2\ep\}$ and for its density $d'$ we have $|d'-d|<\ep.$ 
\end{fact}

\medskip

Given an $(\ep, \delta)$-super-regular pair $H(A, B)$ and a vertex $v\not\in A\cup B,$ we may insert $v$ into $A$
without significantly reducing the ``super-regularity'' of the pair, if $deg(v, B)\ge \delta |B|.$ In fact we may even insert a small linear number of vertices into one vertex class, if these have sufficiently large number of neighbors in the opposite 
vertex class.

\begin{lemma}\label{beilleszt}
Let $H(A, B)$ be an $(\ep, \delta)$-super-regular pair with density $d\ge \delta.$ 
Assume that $0\le s\le \ep^2|A|$ and $0\le t\le \ep^2|B|$ are numbers,  $S_A=\{v_1, \ldots, v_s\}$ and 
$S_B=\{u_1, \ldots, u_t\}$ are sets of vertices such that $S_A\cap A= S_B\cap B=\emptyset,$ and  
$deg(v, B)\ge d |B|$ for every $v\in S_A$ and $deg(u, A)\ge d|A|$ for every $u\in S_B.$ 
Then the new pair $H'(A\cup S_A, B\cup S_B)$ is $(3\ep, \delta-\ep)$-super-regular with density 
$d'=d'(A\cup S_A, B\cup S_B)=d\pm \ep.$  
\end{lemma}

\noindent {\bf Proof:} Let $A'=A\cup S_A$ and $B'=B\cup S_B.$ The lower bounds $(\delta-\ep)|B'|$ and
$(\delta-\ep)|A'|$ for the degrees of vertices in $A'$ and $B',$ respectively, follows easily from the degree bounds for $S_A$ and $S_B.$
Since $S_A$ and $S_B$ are very small compared to $A,$ respectively, $B,$ a simple calculations shows
that $d'$ cannot deviate from $d$ by more than $\ep.$  

Next we verify the $3\ep$-regularity of the new pair. Let $X'\subset A'$ with $|X'|=3\ep |A'|$ and $Y'\subset B'$ 
with $|Y'|=3\ep|B'|,$ by convexity of density (see e.g. in~\cite{KS}) it is enough to consider
subsets of this size. Set $X=X'-S_A$ and $Y=Y'-S_B.$ Clearly, 
$$|X|\ge 3\ep |A'|-\ep^2|A|\ge \ep(3-\ep)|A'|>\ep |A|,$$ and similarly, 
$$|Y|\ge 3\ep |B'|-\ep^2|B|\ge \ep(3-\ep)|B'|>\ep |B|.$$
We need upper and lower bounds for the number of edges between $X'$ and $Y'.$ For the lower bound
we can use the $\ep$-regularity of the original pair $H(A, B),$ since $X$ and $Y$ are sufficiently large:
$$e(X', Y')\ge (d-\ep)|X|\cdot |Y|\ge (d-\ep)\ep^2(3-\ep)^2|A'|\cdot |B'|.$$
This implies that $$d'(X', Y')=\frac{e(X', Y')}{|X'| |Y'|}\ge (d-\ep)\frac{\ep^2(3-\ep)^2 |A'|\cdot |B'|}{9\ep^2 |A'||B'|},$$ here we used that $|X'|=3\ep|A'|$ and $|Y'|=3\ep |B'|.$ Simple calculation shows that the latter expression
is larger than $d-3\ep,$ hence, $d'(X', Y')\ge d-3\ep.$

For estimating the upper bound, we assume the worst case, when vertices of $S_A\cup S_B$ have full degree into the opposite
part:
$$e(X', Y')\le e(X, Y)+e(S_A\cap X', Y')+e(X', S_B\cap Y').$$
Set $x=|X'\cap S_A|$ and $y=|Y'\cap S_B|.$ Then, applying $\ep$-regularity for $e(X, Y),$ we have
$$e(X', Y')\le (d+\ep)|X||Y|+x|Y'|+y|X'|.$$ Dividing by $|X'||Y'|$ we obtain that
$$d'(X', Y')\le (d+\ep)\frac{|X||Y|}{|X'||Y'|}+\frac{x}{|X'|}+\frac{y}{|Y'|}\le 
d+\ep+\frac{\ep^2|A|}{3\ep|A'|}+\frac{\ep^2|B|}{3\ep |B'|}\le d+3\ep.$$
This finishes the proof of the lemma. 
\qedf

\medskip

We use the decomposition theorem of the author~\cite{CsBundle} below for replacing the Regularity lemma. 

\begin{theorem}\label{pontfelbontas}
Let $0< \ep <1/10$ be a number, and assume that $d$ is a real number such that $10 \ep ^{1/5}\le d< 1/3.$  
Assume further that $G=(V, E)$ is a balanced bipartite graph on $2n$ vertices with bipartition $V=A\cup B$ and $r(1-\lambda)\le deg(v)\le r$ 
for every $v\in V,$ where $r\ge d^{1/3}n$ and $0\le \lambda\le d^4/2.$ 
If $n>\exp(150 \log (1/d)/\ep^2),$ then there exists a natural number $K\le 8 \frac{d_G}{d}\cdot d^{-100/\varepsilon^2}$ such that $G$ admits the following decomposition:

\begin{itemize}

\item [(i)] $A=A_1\cup \dots \cup A_K \cup A_0$ and $B=B_1\cup \dots \cup B_K \cup B_0,$
where $A_i\cap A_j=B_i\cap B_j=\emptyset,$ whenever $i\neq j,$

\item [(ii)] $|A_0 \cup B _0|< 8\sqrt[3]{d}n,$ 

\item [(iii)] $|A_i|, |B_i| \ge d^{101/\ep^2}\cdot \frac{n}{4}=\Omega\left(n^{1/3}(\log n)^{-1/2}\right)$ for every 
$1\le i\le K,$

\item [(iv)] $||A_i|-|B_i||\le 2\ep^2|A_i|$ for every $1\le i\le K,$

\item [(v)] the bipartite subgraphs $G[A_i, B_i]\subset G$ for $i\ge 1$ are all $(\ep', d_i)$-super-regular, where 
$\ep'\le (64\ep)^{1/5}$ and $d_i\ge d-3\ep.$ 

\end{itemize}

\end{theorem}

We remark, that in~\cite{CsBundle} a slightly stronger result is proved: for $i\ge 1$ the $G[A_i, B_i]$ pairs
are not only super-regular, there is an upper bound for the degrees in such a pair. We do not need this
stronger notion in the present paper.

The sets $A_1, \ldots, A_K$ and $B_1, \ldots, B_K$ are called the {\it non-exceptional clusters} of the decomposition,
and the sets $A_0$ and $B_0$ are the {\it exceptional clusters,} analogously to the decomposition of the Regularity lemma. 
Note, that the exceptional clusters, while can be made small by choosing $d$ to be small, could be much larger than the
non-exceptional clusters. Analogously to the Regularity lemma, sometimes we will call the union of the
$G[A_i, B_i]$ super-regular pairs the {\it reduced graph} of $G.$ The vertices of this reduced graph are the clusters, the edges are the 
$G[A_i, B_i]$ super-regular pairs, these constitute a matching in the reduced graph.
Note, that the $G[A_i, B_i]$ pairs may have only a very small fraction of the edges of $G$ itself.
Hence, the vast majority of the edges of $G$ do not belong to quasirandom subgraphs.
On the other hand, we do not need that $n$ is a tower function of $\ep.$ 

The partitioning in Theorem~\ref{pontfelbontas} is less powerful than that of the Regularity lemma. 
Still, the quasirandomness of the
$G[A_i, B_i]$ pairs, together with the (unstructured) set of remaining edges of $G$ not belonging to these pairs,
can be used in embedding problems. This is the governing idea which enables us to use the
above decomposition theorem for tree embedding.

\subsection{Probabilistic tools}

We use random methods at various points in the paper, and need large deviation bounds for discrete probability distributions.
The following inequality, a generalized version of Chernoff's bound,  is Theorem 2.8 in~\cite{JLR}.

\begin{theorem}\label{Chernoff} 
Assume that $X$ is the sum of $k$ independent indicator random variables: $X=X_1+\ldots+X_k.$ If $0\le \lambda \le 3/2,$ then
$$P(|X-\mathbb{E}[X]|\ge \lambda \mathbb{E}[X])\le 2e^{-\frac{\lambda^2}{3} \mathbb{E}[X]}.$$ 
\end{theorem}

We will also need another inequality, in which we do not assume independence of the variables. It was proved by Hoeffding and also by Azuma, sometimes it is called Azuma's inequality.
A sequence of random variables $X_0, X_1, \ldots$ is a {\it martingale} if 
${\mathbb E}[X_{i+1}| X_0, \ldots,  X_i] = X_i$ for each $i \ge 0.$ 
We have the following important inequality, see e.g.~in~\cite{AS}.

\begin{theorem}[Azuma-Hoeffding inequality]\label{azuma}
Assume that the sequence $X_0, X_1, \ldots $ is a martingale, and let $\sigma_i> 0$ for all $i\ge 1.$
If $|X_i-X_{i-1}|\le \sigma_i$ for each $i\ge 1$ and $a>0$ is a real number, then for each $n\ge 1$ we have 
$$P(|X_n-X_0|\ge a)\le2 e^{-a^2/2\sigma^2}$$
where $\sigma^2=\sum_{i=1}^t \sigma^2_i.$ 
\end{theorem}


\section{Preprocessing of $G$}

In order to prove Theorem~\ref{fa}, we need to preprocess the host graph $G$ and the tree $T$ as well. In this section we 
focus on $G.$

First we find a bipartite $r$-regular spanning subgraph of $G$ in subsection~\ref{bipreg}. This enables us to use Theorem~\ref{pontfelbontas}, and to construct vertex disjoint quasirandom pairs 
$G[A_1, B_1],\ldots, G[A_K, B_K],$ which together cover almost every vertex of $G.$ 
Since we want to embed a spanning tree, the exceptional vertices of $V_0=A_0\cup B_0$ will be inserted
into the non-exceptional clusters $A_1, \ldots, A_K$ and $B_1, \ldots, B_K.$
A few vertices will also change cluster in order to make sure that every pair in the decomposition is balanced. This is not an easy
task, it is the subject of subsections~\ref{reloclem} and~\ref{balanceproc}. 
Our main goal in this section is to prove the statement below.
 
\begin{prop}\label{Gdecomp}
Assume that $\ep, d, \nu$ and $\gamma$ are numbers such that $0<\ep\ll d\ll\nu \ll \gamma \ll 1.$ Let
$G=(V, E)$ be an $n$-vertex, $\gamma$-non-extremal graph having minimum degree $\delta(G)\ge (1/2-\nu)n,$ where $n\ge n_0=n_0(\nu).$ Then one can divide $V$ into the 
disjoint sets $A$ and $B$ with $|A|=|B|=\lfloor n/2 \rfloor$ and possibly an extra vertex $v_0$ such that
the following hold. The subsets $A$ and $B$ are decomposed into disjoint clusters: 
$A=\widehat{A}_1\cup \ldots \cup \widehat{A}_K$ and $B=\widehat{B}_1\cup \ldots \cup \widehat{B}_K,$ such that for every $i\in [K]$ we have:

\begin{enumerate}
\item $m_i=|\widehat{A}_i|=|\widehat{B}_i|\ge d^{101/\ep^2}\cdot \frac{n}{4},$ 
\item $\widehat{A}_i=A'_i\cup A''_i,$ $A'_i\cap A''_i=\emptyset$ and $\widehat{B}_i=B'_i\cup B''_i,$ $B'_i\cap B''_i=\emptyset,$
\item $|A''_i|, |B''_i| \le \gamma^4 m_i$
\item $G[A'_i, B'_i]$ is a $(2\ep, d/3)$-super-regular pair, 
\item if $v\in A''_i,$ then $deg(v, B'_i)\ge \gamma^3|B'_i|/3,$ and $deg(w, A'_i)\ge \gamma^3|A'_i|/3$ 
for every $w\in B''_i.$
\end{enumerate}
\end{prop}

Observe, that the proposition does not claim that the $G[\widehat{A}_i, \widehat{B}_i]$ pairs are super-regular.
Instead, we have a large super-regular sub-pair of $G[\widehat{A}_i, \widehat{B}_i]$ 
(large, since the clusters $|A'_i|, |B'_i|\ge (1-\gamma^4)m_i$), and vertices in 
the {\it irregular} parts $A''_i$ and $B''_i$ have large degrees to the opposite side. We call $\bigcup_{i} (A''_i\cup B''_i)$ the
set of {\it irregular vertices}. The irregular vertices need a special care when proving Theorem~\ref{fa}. 

Proposition~\ref{Gdecomp} is a general result, we think its use is not restricted to tree embedding, it could be 
applied to other problems. 

\subsection{Finding a bipartite $r$-regular spanning subgraph of $G$}\label{bipreg}

We begin with a structural fact about $\gamma$-non-extremal graphs. Recall, that $G$ is a $\gamma$-non-extremal graph with 
$\delta(G)\ge (1/2-\nu)n,$
where $\nu \ll \gamma.$ 
For every vertex $v\in V$  we 
define a subset of $V$: $$S_v=\{u\in V: deg(u, \overline{N(v)})< \gamma n/2\},$$
where $\overline{N(v)}$ denotes the complement of $N(v).$

\begin{lemma}\label{NonExt}
Let $v\in V$ be any vertex with $deg(v)\le n/2+\gamma n/2.$ Then $|S_v|\le n/2-\gamma n/2.$ 
\end{lemma}

\noindent {\bf Proof:} Suppose on the contrary that $|S_v|>n/2-\gamma n/2.$ We count the number of edges between $S_v$ and $\overline{N(v)}.$
First observe, that $$e(S_v, \overline{N(v)})< \frac{\gamma n}{2} |S_v|\le \gamma n^2/2$$ by the definition of $S_v.$ 

Next, if necessary, we add extra vertices to $S_v$ and $\overline{N(v)}$ so that the resulting new sets both have
at least $n/2$ vertices. If $S_v$ or $\overline{N(v)}$ has at least $n/2$ vertices, we leave it as is.
For complementing $S_v$ we need less than $\gamma n/2$ new vertices using our assumption on the cardinality of $S_v.$ Similarly, we need at most 
$\gamma n/2$
new vertices for $\overline{N(v)},$ since $deg(v)\le n/2+\gamma n/2.$ Denote the new sets we have just obtained by $S'_v$ and $\overline{N(v)}'.$

By $\gamma$-non-extremality of $G,$ we have at least $\gamma n^2$ edges going in between $S'_v$ and $\overline{N(v)}'.$ 
On the other hand, with the newly added vertices we could increase the number of edges  
between $S'_v$ and $\overline{N(v)}'$ by at most $2\gamma \frac{n}{2} \cdot \frac{n}{2}=\gamma n^2/2.$ Hence,
we can add at most $\gamma n^2/2$ edges to the strictly less than $\gamma n^2/2$ edges that were supposedly present between $S_v$ and 
$\overline{N(v)},$ so the total number of edges between $S'_v$ and $\overline{N(v)}'$ is less than $\gamma n^2$ -- thus we arrived at a contradiction.
\qedf

\medskip

Next we randomly split $V$ into two sets, $A$ and $B,$ of sizes $|A|=|B|=n/2,$ if $n$ is even. If $n$ is odd, we 
set aside an arbitrary vertex, denoted by $v_0,$ before the random splitting. That vertex will be inserted back to $G$ at the
end of the preprocessing. 

The random splitting procedure goes as follows. In the beginning, $A$ and $B$ are empty sets. For every $v\in V$ we flip a coin, independently from other choices. If the outcome is heads, we add $v$ to $A,$ otherwise
$v$ is added to $B.$ 
If $|A|\neq |B|,$ we will make them equal as follows. Say, that after the random splitting $|A|>|B|.$ Then we pick $(|A|-|B|)/2$ vertices of $A$ arbitrarily, and relocate them into $B.$ Note that with high probability the number of relocated vertices is at most $O(\sqrt{n\log n}).$ 

The proof of the lemma below is implied by Theorem~\ref{Chernoff}, we leave the details for the reader.

\begin{lemma}\label{VelTul}
For every $v\in V$ we have $$deg(v, A), deg(v, B)\ge \frac{deg(v)}{2}-\nu n \ge n/4-2\nu n$$ with high probability. Furthermore, 
if $deg(v)\le n/2+\gamma n/2,$ then the following properties also hold with high probability:
\begin{itemize}

\item[(i)] $|S_v\cap A|, |S_v\cap B| \le n/4 -\gamma n/5,$ and

\item[(ii)] if $w\not\in S_v,$ then $deg(w, \overline{N(v)}\cap A), deg(w, \overline{N(v)}\cap B)\ge \gamma n/5.$

\end{itemize}
\end{lemma}

The following lemma shows that for spanning random bipartite subgraphs a certain form of non-extremality is inherited with high probability. 

\begin{lemma}\label{ParosNonExtr}
Let $G=(V, E)$ be a $\gamma$-non-extremal graph on $n$ vertices with minimum degree $\delta(G)\ge (1/2-\nu)n.$ Divide $V$ randomly into two parts, $A$ and $B,$ having equal cardinality $n/2,$ as described above. Then the following holds with high probability: 
for every $X\subset A,$ $Y\subset B$ with $|X|=|Y|=n/4$ we have $e(X, Y)\ge \gamma^2n^2/50.$
\end{lemma}

\noindent {\bf Proof:} Assume, that $X\subset A$ and $Y\subset B$ both have $n/4$ vertices, and assume on the 
contrary, that $e(X, Y)< \gamma^2n^2/50.$ Then by averaging there exists a vertex $v\in X$ such that $deg(v, Y)\le \gamma^2n/12.$ This implies that $deg(v)\le n/2+\gamma n/2.$ Moreover, using Lemma~\ref{VelTul}, we must have $|\overline{Y}-N(v)|\le 2\nu n +\gamma^2n/12.$ Hence, using the definition of the set $S_v$ and Lemma~\ref{VelTul}, $A$ has at least $n/4+\gamma n/5$ such vertices
which all have at least $\gamma n/5-2\nu n-\gamma^2n/12\ge \gamma n/10$ neighbors in $Y.$ Since $X$ must contain at least $\gamma n/5$ from them, the number of edges between $X$ and $Y$ is more than $\gamma^2n^2/50,$ as desired.
\qedf

\medskip

Next we show that the induced subgraph $G[A, B]$ has an $r$-regular spanning subgraph with $r= \nu n.$ 
This is done by finding $r$ edge-disjoint perfect matchings in $G[A, B].$ After finding a 1-factor, we delete its
edges from $G[A, B],$ and look for another one. Thus, during this procedure we decrease the degrees
of the vertices. It is easy to see that for the existence of the $r$ edge-disjoint 1-factors it is sufficient to prove the following.

\begin{lemma}\label{1Fact}
Assume, that the edges of less than $r=\nu n$ edge-disjoint perfect matchings were deleted from $G[A, B].$ Denote $G'[A, B]$
the subgraph what is left. Then $G'[A, B]$ has a 1-factor.
\end{lemma}

\noindent {\bf Proof:} We will check the K\H onig-Hall conditions in three steps. 

In the first step, using Lemma~\ref{VelTul}, we have that the minimum degree in $G'[A, B]$ is at least 
$n/4-2\nu n-r=n/4-3\nu n.$ Hence, for every $A'\subset A$ with $|A'|\le n/4-3\nu n$ we have that 
$|N(A')|\ge n/4-3\nu n\ge |A'|.$ 

In the second step we assume that $A'\subset A$ with $n/4-3\nu n<|A'|\le n/4+3\nu n,$ and show that $|N(A')|> n/4+3\nu n.$ Observe first that
if there is a vertex $v\in A'$ with $deg(v) > n/2+\gamma n/2,$ then by Lemma~\ref{VelTul}, $$deg(v, B)\ge n/4+\gamma n/4-\nu n>n/4+3\nu n,$$
implying $|N(A')|>n/4+3\nu n.$

Next we assume that there is a vertex $v\in A'$ with $deg(v)\le n/2+\gamma n/2.$ Lemma~\ref{VelTul} implies that 
$|S_v\cap A|\le n/4 -\gamma n/5.$ Since $$|S_v\cap A|\le n/4-\gamma n/5 < n/4 -3\nu n \le |A'|,$$ 
there exists a vertex $w\in A'-S_v.$ Hence, by Lemma~\ref{VelTul} we have 
$$|N(A')|\ge deg(v, B)+deg(w, \overline{N(v)}\cap B)\ge n/4-3\nu n+\gamma n/5>n/4+3\nu n.$$

Finally, in the third step we may assume that $|A'|\ge n/4-3\nu n +1.$ By Lemma~\ref{VelTul}
every vertex of $B$ will have at least one neighbor in $A',$ hence, in this case $N(A')=B,$ finishing
the proof of the lemma.
\qedf

\medskip

As we discussed earlier, the above lemma immediately implies the following.

\begin{cor}\label{rFact}
The induced subgraph $G[A, B]$ has an $r$-regular spanning subgraph $G'[A, B]$ with $r=\nu n.$ 
\end{cor}

\medskip

This corollary enables us to use Theorem~\ref{pontfelbontas} with parameters $\ep_0$ and $d_0$ such that $\ep_0^{1/5}\le d_0/10$ and 
$d_0^{1/3}\le \nu/10.$ 

\begin{cor}\label{felbontas2}
There exists a decomposition of $A$ into the disjoint sets $A_0, A_1, \ldots, A_K,$ and similarly, a decomposition of $B$ into the disjoint sets
$B_0, B_1, \ldots, B_K,$ such that the bipartite subgraphs $G[A_i, B_i]$ are $(\ep, d_i)$-super-regular pairs for every $i\ge 1,$ where $\ep\le (64\ep_0)^{1/5}$
and $d_i\ge d_0-3\ep$ for every $i\in [K].$ 
\end{cor}

\medskip

We will refer to the sets $A_1, \ldots, A_K,$ and  $B_1, \ldots, B_K$ as {\it non-exceptional clusters},
and $A_0, B_0$ are the {\it exceptional clusters}.

We remark that 
the parameters $\nu$ and $\gamma$ are absolute constants, hence, the numbers $\ep, d$ and $K$ are bounded. 

\medskip

\subsection{A lemma for relocating vertices of $G$}\label{reloclem}

During the embedding of $T$ we may need to relocate a small number of vertices, that is, some vertices have to
change their clusters. 
Given a vertex $v\in V$ and a cluster $A_i$ (where $1\le i\le K$) we say that the $(v, u, w)$ triple is a {\it relocating
$v-A_i$ path}, if there exists $1\le s, t\le K$ such that the following conditions hold:
\begin{enumerate}

\item $|\{u, v, w\}|=3$

\item $deg(v, B_s)\ge \gamma^3 |B_s|$

\item $u\in A_s$ and $deg(u, B_t)\ge \gamma^3 |B_t|$

\item $w\in A_t$ and $deg(w, B_i)\ge \gamma^3 |B_i|$

\end{enumerate}

If we have a relocating $v-A_i$ path $(v, u, w),$ then the relocation is done as follows: first, we put $v$
into the cluster $A_s,$ then $u$ is relocated from $A_s$ into the cluster $A_t,$ and finally, we place $w$ into $A_i$ from $A_t.$ Note, that this way the cardinalities of the clusters $A_s$ and $A_t$ do not change, while the size
of $A_i$ increases by one. 
The definition of a relocating $v-B_i$ path is very similar, one only has to exchange the letters ``$A$'' and ``$B$'' in the above conditions. 

We say, that $v$ is the {\it first vertex} in the $(v, u, w)$ triple, $u$ is the {\it second vertex,} and $w$ is the {\it third vertex.} 
The triples $(v, u_1, w_1)$ and $(v, u_2, w_2)$ are
{\it disjoint} relocating paths, if $|\{u_1, w_1, u_2, w_2\}|=4,$ that is, if the second and third vertices are different.

\begin{lemma}\label{reloc} Let $v\in V$ be any vertex, and $A_s, B_t \subset V-V_0$ be two clusters, where $1\le s, t \le K.$
Then there are at least $\gamma^6n/30$ disjoint $v-A_s$ relocating paths and similarly, at least
 $\gamma^6n/30$ disjoint $v-B_t$ relocating paths.
\end{lemma}

\noindent {\bf Proof:} Without loss of generality, we will only consider $v-A_1$ relocating paths.
Theorem~\ref{pontfelbontas} implies that $n/2-8d^{1/3}n\le \sum_{i=1}^K|A_i|\le n/2$ and 
similarly, $n/2-8d^{1/3}n\le \sum_{i=1}^K|B_i|\le n/2.$ Using Lemma~\ref{VelTul} we have, that $$deg(v, \bigcup_{i=1}^KA_i), deg(v, \bigcup_{i=1}^KB_i) \ge n/4-2\nu n-8d^{1/3}n.$$ Set $\eta = 2\nu +8d^{1/3},$
so the number of neighbors of any vertex is at least $n/4-\eta n$ in $\cup_{i=1}^KA_i$
and in $\cup_{i=1}^KB_i.$ 
Let $I=\{i: deg(v, B_i)> \gamma^3|B_i|\}.$ Set $B_I=\bigcup_{i\in I}B_i$ and $A_I=\bigcup_{i\in I}A_i.$

\begin{claim}\label{xalso}
We have $|B_I|\ge (\frac{1}{2}-2\gamma^3)\frac{n}{2}.$
\end{claim}

\noindent {\bf Proof:} (of the claim) Let $\alpha=|B_I|/(n/2).$ For estimating $\alpha$ from below, we assume the worst case, that is, 
$v$ has full degree into every $B_i$ for $i\in I,$ and has precisely $\gamma^3|B_t|$ neighbors in $B_t,$ whenever $t\not\in I.$ 

Then we have the following inequality:
$$\alpha\frac{n}{2}+\gamma^3(1-\alpha)\frac{n}{2}\ge \frac{n}{4}-\eta n.$$ Observe, that here we have $(1-\alpha)n/2$ for 
$\sum_{i\not\in I}|B_i|$ -- it is easy to see, that using the upper bound $n/2$ for the total number of vertices
in $\bigcup_{i=1}^KB_i$ results in a smaller, therefore valid, lower bound for $\alpha.$ After rearranging we have
$$\alpha(1-\gamma^3)\ge \frac{1}{2}-2\eta -\gamma^3,$$ implying that $$\alpha\ge \frac{\frac{1}{2}-2\eta -\gamma^3}{1-\gamma^3}\ge \frac{1}{2}-2\gamma^3,$$
since $\eta<\gamma^3.$ \qedf

\medskip

\begin{claim}\label{xalso2}
We have $|A_I|\ge (\frac{1}{2}-3\gamma^3)\frac{n}{2}.$
\end{claim}

\noindent {\bf Proof:} (of the claim) By Theorem~\ref{pontfelbontas}, $||A_j|-|B_j||\le 2\ep^2 |A_j|$ for every 
$1\le j\le K.$ Hence, $$1-\ep \le \frac{|B_j|}{|A_j|}\le 1+\ep.$$ From this we obtain 
$$\frac{1}{1+\ep}|B_I|\le |A_I|.$$ Since 
$(1+\ep)(\frac{1}{2}-3\gamma^3)< \frac{1}{2}-2\gamma^3,$ the claim follows. \qedf

\medskip

We need a new set of indices: let 
$$J=\{j: \exists S\subset A_j \ {\rm with \ } |S|= \gamma^3|A_j| \ {\rm such \ that \ } deg(w, B_1)> {\gamma^3}|B_1| \ {\rm for \ every \ } w\in S \}.$$ Set $A_J=\bigcup_{j\in J}A_j$ and $B_J=\bigcup_{j\in J}B_j.$

\begin{claim}\label{yalso}
We have $|A_J|\ge (\frac{1}{2}-3\gamma^3)\frac{n}{2}$ and $|B_J|\ge (\frac{1}{2}-4\gamma^3)\frac{n}{2}.$
\end{claim}
\noindent {\bf Proof:} (of the claim) Let $\alpha=|A_J|/(n/2).$ We use assumptions similar to the ones in Claim~\ref{xalso}: if $j\in J,$ then every vertex $A_J$ has full degree into $B_1,$ and for
$t\not\in J,$ $A_t$ has precisely $\gamma^3|A_t|$ vertices that have full degree into $B_1,$ the rest, 
$(1-\gamma^3)|A_t|,$ have
precisely $\gamma^3|B_1|$ neighbors in $B_1.$ Then we obtain the following inequality:
$$\alpha\frac{n}{2}|B_1|+(1-\alpha)(\gamma^3+(1-\gamma^3){\gamma^3})\frac{n}{2}|B_1|\ge |B_1|(\frac{n}{4}-\eta n).$$ Dividing by
$|B_1|\cdot \frac{n}{2}$ and then rearranging we get
$$\alpha(1-2\gamma^3+\gamma^6)\ge \frac{1}{2}-2\eta -2\gamma^3+\gamma^6.$$ Since $\gamma\le 1/100,$ this implies
the claimed bound for $\alpha.$ The lower bound for $|B_J|$ can be obtained similarly to the proof of Claim~\ref{xalso2}, we
leave the details for the reader.
\qedf

\medskip

The previous claims help us to use bipartite non-extremality as follows.

\begin{claim}\label{elszam}
We have $e(A_I, B_J)\ge \gamma^2n^2/100.$
\end{claim} 

\noindent {\bf Proof:} (of the claim) We begin with adding $n/4-|A_I|\le 2\gamma^3 n$ vertices to $A_I,$ if necessary, and similarly, $n/4-|B_J|\le 2\gamma^3 n$ vertices to $B_J$ in order to achieve that these subsets have cardinality at least $n/4.$ These are the {\it extra} vertices.
The upper bounds for the number of extra vertices follow from Claim~\ref{xalso2} and Claim~\ref{yalso}, respectively.

Using Lemma~\ref{ParosNonExtr} we have at least $\gamma^2n^2/50$ edges between the two subsets.
The number of edges which have at least one extra vertex endpoint is at most $2\cdot 2\gamma^3 n \cdot n/4=\gamma^3 n^2.$
Since $\gamma \le 1/100,$ we proved what was desired.
\qedf

\medskip

\begin{claim}
There are at least $\gamma^2n/50$ such vertices in $A_I$ which all have at least $2\gamma^3|B_J|$ neighbors in $B_J.$ 
\end{claim}

\noindent {\bf Proof:}
We assume that $|A_I|, |B_J|\le n/4.$ If not, one can arbitrarily leave out vertices from them. Let $\alpha|A_I|$ denote the number of 
vertices in $A_I$ which all have at least $2\gamma^3|B_J|$ neighbors in $B_J.$ We have the following inequality:
$$\alpha|A_I|\cdot |B_J|+(1-\alpha)2\gamma^3|A_I|\cdot |B_J|\ge \gamma^2 \frac{n^2}{100}\ge \gamma^2 \frac{4}{25}|A_I|\cdot |B_J|.$$
Dividing by $|A_I|\cdot |B_J|,$ after rearranging we have $$\alpha(1-2\gamma^3) \ge \frac{4}{25}\gamma^2-2\gamma^3.$$
Hence, $\alpha\ge \gamma^2/10.$ Using that $|A_I|\ge (1/2-3\gamma^3)n/2,$ the claimed bound follows.
\qedf

\medskip

Given the set $J$ and a vertex $u$ we define a new set of indices $J(u)\subseteq J$ as follows:
$$J(u)=\{ j : j\in J \ {\rm and \ } deg(u, B_j)>  \gamma^3|B_j|\}.$$ Set $B_{J(u)}=\bigcup_{j\in J(u)}B_j$ and $A_{J(u)}=\bigcup_{j\in J(u)}A_j.$ 

\smallskip

\begin{claim}
Let $u\in A_I$ be a vertex which has at least $2\gamma^3|B_J|$ neighbors in $B_J.$ Then $|B_{J(u)}|\ge \gamma^3|B_J|/2$ 
and $|A_{J(u)}|\ge \gamma^3|A_J|/3.$
\end{claim}

\noindent {\bf Proof:}
Let $\alpha=|B_{J(u)}|/|B_J|.$ The following inequality is satisfied:
$$\alpha|B_J|+(1-\alpha)\gamma^3|B_J|\ge 2\gamma^3|B_J|.$$ Dividing by $|B_J|$ and rearranging gives
$$\alpha\ge  \frac{\gamma^3}{1-\gamma^3}>\frac{\gamma^3}{2},$$ as desired.
The second part of the statement follows easily, as in Claim~\ref{xalso2}, from the fact that $|A_j|$ and $|B_j|$ may differ only in a small proportion
for every $j.$
\qedf

\medskip

After these preparations we are ready to prove the lemma. Consider the triples of the type $(v, u, w),$ where 
$u\in A_i\subseteq  A_I$ and $w\in A_j\subseteq A_{J(u)},$ so $i\in I$ and $j\in J(u)\subseteq J.$ 
By the definition of the index sets $I, J$ and $J(u),$ $A_I$ has at least $\gamma^2n/50$ vertices $u$ with 
$deg(u, B_j)\ge \gamma^3|B_j|$ which all can be second vertices. For a given second vertex $u,$ if $j\in J(u),$ the cluster 
$A_j$ has at least $\gamma^3|A_j|$ 
such vertices $w$ for which $deg(w, B_1)\ge \gamma^3|B_1|.$ Since 
$$|A_{J(u)}|\ge \gamma^3|A_J|/3\ge \frac{\gamma^3}{3}\left(\frac{1}{2}-3\gamma^3\right)\frac{n}{2}>\gamma^3\frac{n}{15},$$
we have that for any given $v,$ one can choose at least $\gamma^2n/50$ second vertices, and for any such 
second vertex $u$ one can choose at least $\gamma^3|A_{J(u)}|> \gamma^6 {n/15}$ third vertices. This means
that there are at least $\gamma^5n^2/750$ relocating triples for $v.$ Of course, the disjointness requirement is not satisfied. 

We can find the disjoint triples from the set of all triples using a greedy algorithm. In the beginning every vertex
in the triples will be {\it available}. The key observation is that whenever we 
choose an available second vertex $u$ and an available third vertex $w,$ both $u$ and $w$ will be deleted from the set 
of available vertices. This procedure guarantees the disjointness of the triples. 

Using that the number of possible second vertices is at least $\gamma^2n/50 > 2 \gamma^6n/30,$ and for each
we have at least $\gamma^6n/15=2\gamma^6n/30$ possibilities for a third vertex, we cannot get stuck before finding at least
$\gamma^6n/30$ disjoint triples. This finishes the proof of the lemma.
\qedf

\subsection{Balancing procedure}\label{balanceproc}

Our goal in this section is to distribute all vertices in $V_0$ among the non-exceptional clusters, and if necessary, a small 
proportion of vertices in $V-V_0,$ such that when the procedure have finished, all pairs in the (new) decomposition are balanced. 
The main tool for this is the relocation algorithm given by Lemma~\ref{reloc}, which will be applied $O(d^{1/3}n)$ times 
during the balancing. We remark, that towards the end of the embedding, Lemma~\ref{reloc} will also play an important role,
in a slightly different kind of balancing procedure.

We require that no cluster is ``overused'', more precisely, the proportion of vertices which participate in any relocating triple during
the balancing must not be larger than $\gamma^4$ in any cluster.

Recall, that in an applicaton of Lemma~\ref{reloc} three vertices change their locations, and the second and the third one of
the relocating triple belong 
to some non-exceptional cluster. At any point in time during the balancing, for every $i\in [K],$ we denote by $A'_i\subseteq A_i$ and $B'_i\subseteq B_i$ the sets of those vertices, which {\it belonged} to 
$A_i,$ respectively, $B_i,$ in the beginning, and $A''_i, B''_i$ denote the sets of those vertices, which were {\it added} to the clusters
of the $i$th pair. The clusters in the $i$th pair are denoted by $\widehat{A}_i=A'_i\cup A''_i$ and $\widehat{B}_i=B'_i\cup B''_i$
during the balancing.

These sets may change dynamically. 
In the beginning, $\widehat{A}_i=A_i=A'_i$ and $\widehat{B}_i=B_i=B'_i$ and 
therefore $A''_i=B''_i=\emptyset$ for every $i\in [K].$ At any point in time we have $A'_i=\widehat{A}_i\cap A_i,$
$A''_i=\widehat{A}_i-A_i,$ $B'_i=\widehat{B}_i\cap B_i,$ and $B''_i=\widehat{B}_i-B_i.$ 
When the balancing algorithm finishes, we have $$V_0=A_0\cup B_0\subseteq \bigcup_{i\in [K]} (A''_i\cup B''_i).$$

Recall, that Lemma~\ref{reloc} provides at least $\gamma^6n/30$ disjoint relocating triples for every $v\in V $ and every
non-exceptional cluster. The number of vertices to be relocated is in the order $O(d^{1/3}n),$ which is 
much larger than $dn.$ Hence,
we have to be careful, otherwise we may lose the minimum degree in the originally 
super-regular pairs. For this reason we will only pick relocating triples that contain second and third vertices
which are ``spread out'', we pick them from random subsets of the clusters. 

Note that the quasirandomness is destroyed, 
since relatively many vertices are incorporated into the regular pairs, even though $G[A'_i, B'_i]$ remains $2\ep$-regular. Still, it is possible to obtain the quasirandomness back for 
large subpairs, but for achieving it we have to work hard later during the embedding, in Section~\ref{cover}.

More precisely, we do the following:
for every vertex in the non-exceptional clusters we flip a coin, independently from other choices, 
randomly dividing the non-exceptional clusters into two subsets, depending on the outcome of the coin flip:
for every $i\in [K]$ we let $A_i=A_i^{(1)}\cup A_i^{(2)}$ and $B_i=B_i^{(1)}\cup B_i^{(2)},$ where 
$A_i^{(1)}\cap A_i^{(2)}=B_i^{(1)}\cap B_i^{(2)}=\emptyset.$ 

Set $S=\bigcup_i (A_i^{(1)}\cup B_i^{(1)}).$ By Chernoff's inequality the random subsets of a 
cluster are about the same size, and we also have other nice properties, in particular, we can maintain
large minimum degrees during the balancing into the sets $A_i^{(2)}$ and $B_i^{(2)}.$ This will be discussed in more detail later.

Whenever we need to choose a relocating triple, we will pick one among those triples that have their second
and third vertices in $S.$ Since the probability that both a second and a third
vertex of a triple belongs to some sub-cluster in $S,$ is 1/4, we expect about at least
$\gamma^6n/120$ triples for every $v\in V$ and cluster, which remain for relocation. Here we used the disjointness of
the triples. By Chernoff's inequality, for every $v\in V$ and sub-cluster in $S$ there will remain at least $\gamma^6n/150$ relocating triples
with high probability. 

Assume that we are after applying Theorem~\ref{pontfelbontas} for $G[A, B].$ 
Given a pair $G[\widehat{A}_i, \widehat{B}_i]$ we say that it has a {\it surplus,} if $|\widehat{A}_i|>|\widehat{B}_i|,$
and has a {\it deficiency,} if $|\widehat{A}_i|<|\widehat{B}_i|,$ otherwise we call the pair {\it balanced.} Note, that this definition
works throughout the balancing algorithm. 
During the algorithm we call a triple $(v, u, w)$ {\it available} in the $t$th step, if the previous $t-1$
triples have not used $v, u$ and $w.$ 

\medskip

After these preparations the description of the {\it Balancing algorithm} is as follows.

\begin{enumerate}

\item If $A_0$ is non-empty and there is a pair $G[\widehat{A}_i, \widehat{B}_i]$ 
with deficiency, we pick an arbitrary vertex $v\in A_0,$ and then randomly, uniformly a relocating 
$v-\widehat{A}_i$ triple 
among the available ones with 
first vertex $v$ and second and third vertices belonging to $S,$ and apply Lemma~\ref{reloc} with vertex $v$ 
and cluster $\widehat{A}_i.$
Delete $v$ from $A_0.$ Maintain the sets (cluster and sub-clusters) used in the relocation accordingly.

\item If $A_0$ is non-empty and there are no pairs with deficiency, then we pick a pair $G[\widehat{A}_i, \widehat{B}_i]$ with minimum surplus (this could be a balanced pair), an arbitrary $v\in A_0,$ and a random available
$v-\widehat{A}_i$ relocation triple, and apply Lemma~\ref{reloc} with vertex $v$ 
and cluster $\widehat{A}_i.$ Delete $v$ from $A_0.$ Maintain the sets (clusters and sub-clusters) used in the relocation accordingly.

\item If $A_0=\emptyset$ and $B_0\neq \emptyset,$ then there exists at least one pair with a surplus. 
Let $G[\widehat{A}_i, \widehat{B}_i]$ be such a pair. Pick an arbitrary $v\in B_0,$ choose randomly an available
$v-\widehat{B}_i$ triple, and apply Lemma~\ref{reloc} with vertex $v$ 
and cluster $\widehat{B}_i.$ Delete $v$ from $B_0.$ Maintain the sets (clusters and sub-clusters) used in the relocation accordingly.

\item If $A_0=B_0=\emptyset$ and there is a pair $G[\widehat{A}_i, \widehat{B}_i]$ with a surplus and a pair
$G[\widehat{A}_j, \widehat{B}_j]$ with a deficiency, then pick an arbitrary vertex $v\in A_i^{(1)}$ as a first vertex,
and randomly choose a relocating triple among the available ones with second and third vertices 
belonging to $S.$ Apply Lemma~\ref{reloc} with vertex $v$ 
and cluster $A_j.$ Maintain the sets (clusters and sub-clusters) used in the relocation accordingly.

\item If $A_0=\emptyset,$ and there is no pair with surplus, then all pairs are balanced and $B_0=\emptyset,$
since $|A|=|B|.$ We stop.

\item Repeat the above steps until all pairs become balanced, and $A_0, B_0$ become empty.
 
\end{enumerate}
 
\begin{claim}\label{ballepes}
The number of relocations is at most $|V_0|+2\ep^2n\le 9 d^{1/3}n.$
\end{claim}

\noindent {\bf Proof:} Recall, that $|V_0|=|A_0|+|B_0|< 8 d^{1/3}n$ by Theorem~\ref{pontfelbontas}. The total number of relocations done in steps \#1 and \#2 is $|A_0|.$
 One can perform step \#3 precisely $|B_0|$ times.
It is clear, that if $A_0=B_0=\emptyset,$ then the cluster sizes in the $i$th pair may differ by at most $2\ep^2|A_i|,$
this value follows from $(iv)$ of Theorem~\ref{pontfelbontas}. 
Hence, in step \#4 the total number of relocations is bounded above by $\frac{1}{2}\sum_i ||A_i|-|B_i||\le 2\ep^2n.$
It is easy to see that if the condition of step \#5 is satisfied, then all pairs $G[\widehat{A}_j, \widehat{B}_j]$ must
be balanced, hence the procedure stops.
 Since $\ep \ll d,$ we proved what was desired.  
\qedf

\medskip

\begin{claim}\label{kisvesszo}
For every $i\in [K]$ we have $|A_i''|, |B_i''|\le \gamma^4 |\widehat{A}_i|$ with high probability.
\end{claim}

\noindent {\bf Proof:} Fix an $i\in [K].$ It is clear that $A_i$ and $B_i$ may not contain more than $2\ep^2|A_i|$
first vertices from relocating triples, these are the ones used in step \#4. 

The second and third vertices of relocating triples were randomly chosen from the available ones.
Let $\ell$ denote the number of relocations, and for $j=1, \ldots, \ell,$ define the random variable $X_j$ to be 1, if $A^{(1)}_i$ contains a vertex
from the triple of the $j$th relocation. Since the number of relocating triples is at least $\gamma^6n/150$ 
for every first vertex in the beginning, and one relocation decreases the number of available 
triples for all available first vertices by 2 (the second and the third vertex may destroy 2 triples of an available 
first vertex),  the number of available triples for all first vertices is larger than 
$\gamma^6n/200.$ Here we used that $d^{1/3}\ll \gamma^6.$ 

In Theorem~\ref{pontfelbontas} there is no upper bound for the size of the non-exceptional clusters. This forces us to consider two cases. The first case is when $|A^{(1)}_i|\ge \gamma^6n/200.$ Then even if all
relocating triples use a vertex from $A^{(1)}_i,$ the claim clearly holds. From now on we assume the
second case, when $|A^{(1)}_i|< \gamma^6n/200.$

Let $X=\sum_jX_j,$ then $X$ is the number of second and third vertices of
relocating triples we used during the balancing from $A^{(1)}_i.$

Clearly, for every $j\in [\ell]$ we have $$P(X_j=1)\le \frac{200|A^{(1)}_i|}{\gamma^6n},$$ thus 
$$\mathbb{E}X= O\left(\frac{d^{1/3}|A^{(1)}_i|}{\gamma^6}\right)<\frac{\gamma^4}{2}|A^{(1)}_i|.$$

We define a set of random variables: $Y_j=\mathbb{E}[X| X_1, \ldots, X_{j}]$ for $1\le j\le \ell,$ hence, $Y_{\ell}=X.$ Also, let $Y_0=\mathbb{E}[X].$ 
Observe, that $\mathbb{E}[Y_{j+1}|Y_1, \ldots, Y_j]=Y_j,$ hence, $Y_0, \ldots, Y_{\ell}$ is a martingale process. It is easy to see that $Y_{\ell}=X$ and $\sigma_j=|Y_j-Y_{j-1}|\le 200|A^{(1)}_i|/(\gamma^6n).$ Azuma's inequality asserts that 
$P(Y_{\ell}\ge \mathbb{E}[X]+\lambda) \le 2\exp(-\lambda^2/(2\sigma^2)),$ where $\sigma^2=\sum_{j=1}^{\ell}\sigma_j^2.$ 
Substituting $\lambda=\gamma^4|A^{(1)}_i|/2$ we have that 
$$ P(X\ge \gamma^4|A^{(1)}_i|)\le P(X \ge \mathbb{E}[X]+\gamma^4|A^{(1)}_i|/2)\le 2e^{-\mu n},$$
where $\mu= \gamma^{20}d^{-1/3}/(72\cdot 200^2),$ a constant. It is easy to see that if $n$ is sufficiently large, then the probability in question is less than
$1/n.$ Since the same reasoning works for $B''_i$ as well, we proved what was desired. \qedf

\subsection{Finishing the proof of Proposition~\ref{Gdecomp}}

We are ready to finish the proof. By Claim~\ref{ballepes}, in less than $9d^{1/3}n$ steps we have
$m_i=|\widehat{A}_i|=|\widehat{B}_i|$ for every $i.$  Claim~\ref{kisvesszo} implies that $|A''_i|, |B''_i| \le \gamma^4 m_i$  for every $i.$ Since $A'_i=A^{(2)}_i\cup (A^{(1)}_i-A''_i),$ and $|A''_i|\le \gamma^4 m_i,$ using Chernoff's inequality, $|A'_i|\ge 2|A_i|/3$ with high probability. Similar inequality holds for the cardinality of $B'_i.$ 

Chernoff's inequality also
implies that for every 
$v \in V,$ if $deg(v, A'_i)\ge (d-\ep)|A'_i|,$ then we have $$deg(v, A'_i)\ge deg(v, A^{(2)}) \ge deg(v, A_i)/3$$ and 
similarly, $$deg(v, B'_i)\ge deg(v, B^{(2)}) \ge deg(v, B_i)/3,$$ if  $deg(v, B'_i)\ge (d-\ep)|B'_i|,$ Hence, by definition
of super-regularity, $G[A'_i, B'_i]$ is a $(2\ep, d/3)$-super-regular pair for every $i$ with high probability. By the definition of relocating triples, 
if $v\in A''_i,$ then $deg(v, B'_i)\ge \gamma^3|B'_i|/3,$ and $deg(w, A'_i)\ge \gamma^3|A'_i|/3$ 
for every $w\in B''_i.$  
This finishes the proof of the proposition.
\qedf

\section{Preprocessing of $T$}

Recall, that $T$ is a tree on $n$ vertices with maximum degree $D,$ where $D$ does not depend on $n.$
An important ingredient of the decomposition of $T$ is the following folklore result.

\begin{lemma} \label{split}
Let $\Gamma$ be any tree on $t$ vertices. Then $\Gamma$ has a \emph{split vertex} $x \in V(\Gamma)$ such that it is possible to group the
vertices of $\Gamma-x$ into two forests, $\Gamma_1$ and $\Gamma_2$ such that $t/3 \le v(\Gamma_1), v(\Gamma_2) \le 2t/3$ and there is
no edge connecting $\Gamma_1$ and $\Gamma_2$ in $\Gamma-x.$
\end{lemma}

\smallskip

Let $0<\eta \ll 1$ be a real number. Let us apply Lemma~\ref{split} repeatedly, until each subtree we obtain has at most
$\eta n$ vertices. This way one can arrive at a decomposition of the tree $T$ into the split vertices $s_1, \ldots, s_k,$
and the sub-forests $F_1, \ldots, F_t,$ such that (a) $|F_i|\le \eta n$ for every $i,$ 
(b) $k\le D/\eta,$ and (c) $T$ has no edge between $F_i$ and $F_j,$
whenever $i\neq j,$ while we may have edges between split vertices. Perhaps property (b) is not immediate,
it follows from the fact that even the smallest forest we obtain must have at least $\eta n/D$ vertices.

\smallskip

Let $T'$ denote the {\it smallest subtree} of $T$ that contains all split vertices. This subtree has at most $k$ leaves,
since only a split vertex can be a leaf in $T'.$ Next we {\it mark} the split vertices and also those vertices in
$V(T')$ which have at least 3 neighbors in $T',$ the other vertices of $T'$ are called {\it unmarked}. 

\begin{fact}\label{marked}
The number of marked vertices in $T'$ is at most $2k.$
\end{fact}

\noindent {\bf Proof:} Recall that every leaf of $T'$ is a split vertex. The desired inequality follows from the fact that 
in a tree the number of vertices with degree at least 3 is a lower bound for the number of leaves.
\qedf

\medskip

A path $x_1e_1x_2\ldots x_te_tx_{l+1}$ of length $l$ in $T'$ is called a {\it line} if except possibly the endpoints $x_1$ and $x_{l+1},$ every
vertex of the path has degree two in $T',$ and all the inner vertices $x_2, \ldots, x_l$ are unmarked. 

Let us consider the maximal lines in $T'.$ The endpoints of the maximal lines are marked vertices. If the length of a maximal line is at least 10,
we call it {\it long}, otherwise we call it {\it short}. The inner vertices of the long lines are deleted from $T',$ while we keep every vertex of the short ones. 
Denote the resulting sub-forest of $T'$ by $T_S,$ this is the {\it skeleton} of $T.$

\begin{fact}\label{skeleton}
The skeleton $T_S$ has at most $20k$ vertices.
\end{fact}
\noindent {\bf Proof:} 
By Fact~\ref{marked} the number of maximal lines in $T'$ is at most $2k.$ Hence, the short maximal lines contain a total of at most $18k$ vertices. Taking the marked vertices into account we obtain the desired 
upper bound for $|V(T_S)|.$   
\qedf

\medskip

Clearly, every component of $T-T_S$ has at most $\eta n$ vertices, since $T_S$ contains every split vertex. 
Since the maximum degree of $T$ is $D,$ the number of components in 
$T-T_S$ is at most $20kD\le 20 D^2/\eta.$

\begin{lemma}\label{FokSkeletonba}
If $F$ is a component in $T-T_S,$ then it
may have either one or two neighbors in $T_S.$ 
\end{lemma}

\noindent{\bf Proof:} Observe first, that every component of $T-T'$ is connected to $T'$ by
precisely one edge. We must have at least one edge, since $T$ is connected, and we cannot have 2 or more,
since $T'$ is connected and $T$ is cycle-free. The endpoint of such an edge which belongs to $V(T')$ is either a marked vertex,
or an inner vertex of a maximal line.

Say, that $F$ is a component in $T-T',$ such that its only neighbor in $V(T')$ is a marked vertex $x.$ Then 
$F$ remains a component in $T-T_S,$ and its only neighbor remains $x\in V(T_S).$

If $F_1, \ldots, F_t$ are components in $T-T'$ such that all of them have their neighbors in the same long maximal line $\ell,$
then they will belong to the same component in $T-T_S,$ which also includes all the inner vertices of $\ell.$  
On the other hand, if $F_1, F_2$ are components in $T-T'$ such that their neighbors belong to 
different maximal lines, then no component of $T-T_S$ will contain both, since that would mean
a cycle in $T.$ 

Hence, whenever two or more components in $T-T'$ are united in $T-T_S,$ then these
must have their $T'$-neighbors in the same long maximal line. But then this new, larger component has
precisely two neighbors in $T_S$: the two marked vertices, which are the two endpoints of a long maximal line. 
This finishes the proof of the lemma.
\qedf

\medskip

We summarize the properties of the tree decomposition below.

\begin{prop}\label{Tdecomp}
Let $T$ be a tree on $n$ vertices with maximum degree $D,$ and let $0<\eta<1$ be a real number. Then there exists a 
subforest $T_S\subset T$ with the following properties: 
\begin{enumerate}
\item $T_S$ has at most $20D/\eta$ vertices;
\item every component of $T-T_S$ has at most $\eta n$ vertices;
\item the number of components in $T-T_S$ is at most $20D^2/\eta;$
\item if $F$ is a component in $T-T_S,$ then $1\le |N(F)\cap V(T_S)|\le 2,$ moreover, $e(F, T_S)\le 2;$
\item if a component $F$ in $T-T_S$ has two neighbors $y, y' \in V(T_S),$ then $y$ and $y'$ are the endpoints of a long maximal line in $T',$ they belong to different components of $T_S,$ and their neighbors in $V(F)$ are $x, x'$ with $x\neq x'.$ 
\end{enumerate}
\end{prop}

The following definition will prove to be useful later. Given a rooted tree $\Gamma$ with root $\rho$ we define {\it level sets} of $\Gamma$: for $i\ge 0$ the
$i$th level set, $L_i(\Gamma)$ includes those vertices of $\Gamma$ which are at distance $i$ from $\rho.$
In particular, $L_0(\Gamma)=\{\rho\}.$

Finally, we define the {\it imbalance} of a tree $\Gamma.$ Assume that $\chi$ is a good 2-coloring of $\Gamma,$
$V_1$ denotes the set of vertices colored 1 by $\chi,$ and $V_2$ is the set of vertices colored 2. 
Then we let $$Imb(\Gamma)= \max\{|V_1|, |V_2|\} - \min\{|V_1|, |V_2|\}.$$ If $\Gamma$ is a forest with components
$\Gamma_1, \ldots, \Gamma_t,$ then we let $$Imb(\Gamma)=\sum_{i=1}^tImb(\Gamma_i).$$ Clearly, 
$0\le Imb(\Gamma)< v(\Gamma)$ for every tree or forest $\Gamma.$

This notion plays a crucial role in the proof of Theorem~\ref{fa}, as 
the embedding is done in increasing order of imbalances. 



\section{Further tools for the proof of Theorem~\ref{fa}}

After the preprocessing of $G$ and $T$ we need one further step before we can start the proof of Theorem~\ref{fa}. Our goal is to find an edge preserving bijective mapping, that is, an embedding function
$\Phi: V(T)\longrightarrow V(G).$ We construct the mapping $\Phi$ in several steps, beginning with the skeleton $T_S,$ 
and then extending this partial embedding by finding the images of the components of $T-T_S,$ one by one. Once we have determined $v=\Phi(x)$ for some $x\in V(T)$ and $v\in V(G),$ we will not change $\Phi(x).$

We construct the function $\Phi$ by a randomized algorithm. This algorithm has three phases. 
In the first phase, after preprocessing $G$ and $T,$ we embed the skeleton $T_S.$ 
In the second phase we cover the irregular vertices  (recall, that these are the vertices of $\bigcup_i(A''_i\cup B''_i)$) by components
in $T-T_S.$ 
Finally, in the third phase the vast majority of $T$ is embedded, using the very powerful Blow-up Lemma~\cite{KSSzBl2}.

Before presenting the three phases, we need two lemmas, which play essential roles in the embedding algorithm.
The first of these is used for connecting a subtree to be embedded to the already 
embedded skeleton. The second lemma is used for covering almost all irregular vertices, a small discrepancy can be tolerated. 
We state and prove these lemmas in Sections~\ref{connect} and~\ref{cover}, respectively. Finally, Section~\ref{alg} includes the embedding algorithm and its proof of correctness. 

We need the following hierarchy of the values of the constants used throughout the proof: $$0<\eta\ll \ep \ll d\ll \nu \ll \gamma <1/4,$$ here $\eta$ is the new constant we used for decomposing $T.$

\subsection{Building connections for vertices of $T$}\label{connect}

The embedding of $T$ will begin with the skeleton $T_S.$ After fixing the images of the vertices of $T_S,$
we will distribute the components of $T-T_S$ among the cluster pairs $G[\widehat{A}_i, \widehat{B}_i],$ $i\in [K].$ 
Given a component $F$ in $T-T_S$ we need to build the connection between $F$ and the skeleton. The lemma
below shows how to find this using the non-extremality of $G.$ 

\begin{lemma}\label{connect}
Let $u\in V$ be an arbitrary vertex, $H\subset V$ be a set with $|H| = 3D^3/\gamma,$ and 
$W\subset V-V_0-H$ be a set with $|W|\le \gamma n/10.$ Assume, that $\Gamma$ is rooted tree with
root $\rho$ and maximum degree $D$ such that $L_3(\Gamma)\neq\emptyset$ and 
$L_4(\Gamma)=\emptyset.$
Then $G$ has a copy of $\Gamma$ such that the image of $\rho$ is $u,$ the vertices of $L_3(\Gamma)$
are mapped onto vertices of $H,$ and the images of $L_1(\Gamma)\cup L_2(\Gamma)\cup L_3(\Gamma)$ avoid $W.$ 
\end{lemma}

\noindent {\bf Proof:} Let $N^-(u)=N(u)\cap (V-V_0-W-H),$ then $|N^-(u)|\ge (1/2 -\gamma/4)n.$ 
Next we estimate the number of those vertices in $G$ which have at least $D^3$ neighbors in $H.$
Denote $\alpha n$ the number of those vertices which have {\it  less} than $D^3$ neighbors in $H,$ 
then $G$ has $(1-\alpha)n$ vertices with at least $D^3$ neighbors. Using that the number of edges
incident to vertices of $H$ is at least $|H| (1/2-\gamma/4)n,$ we obtain the following inequality:
$$\alpha n D^3 +(1-\alpha)n|H| \ge \left(\frac{1}{2}-\frac{\gamma}{4}\right)|H| n.$$ Rearranging gives that
$$\alpha \le \frac{|H|}{|H|-D^3}\left(\frac{1}{2}+\frac{\gamma}{4}\right) = 
\frac{3}{3-\gamma}\left(\frac{1}{2}+\frac{\gamma}{4}\right) < \frac{1+\gamma}{2}.$$ Hence, there is a set 
$U\subset V-W-H$
with $|U|\ge (1-\gamma)n/2-\gamma n/10-|H|\ge n/2-3\gamma n/4$ 
such that $deg(a, H)\ge D^3$ for every $a\in U.$ If we add at most $\gamma n/4$ new vertices to 
$N^-(u)$ to make it a set with cardinality precisely $n/2,$
and similarly, if we add at most $3\gamma n/4$ vertices to $U$ in order to obtain a set with precisely
$n/2$ vertices, then we may add at most $(\gamma n/4 + 3\gamma n/4) \cdot n/2=\gamma n^2/2$ 
new edges to the bipartite subgraph $G[N^-(u), U].$

Hence, by $\gamma$-non-extremality of $G,$ we have $$e(N^-(u), U)\ge \gamma n^2-\gamma n^2/2=\gamma n^2/2.$$ This implies that $N^-(u)$ has at least $\gamma n/2>|L_1(\Gamma)|$ such vertices which each has at least $\gamma n/2>|L_2(\Gamma)|$ neighbors in $U.$ Since $|L_3(\Gamma)|<D^3,$ 
using the definition of $U$ we can find the desired copy of $\Gamma,$ which satisfies all requirements.
\qedf

\subsection{Covering the vast majority of the irregular vertices}\label{cover}

For embedding the vast majority of 
$T$ we will use the Blow-up lemma~\cite{KSSzBl2}, but for applying it we need super-regularity.  
Recall, that $\widehat{A}_i=A'_i\cup A''_i$ and similarly, $\widehat{B}_i=B'_i\cup B''_i,$ for $1\le i\le K,$ 
where $A'_i=A_i\cap \widehat{A}_i$ and $B'_i=B_i\cap \widehat{B}_i,$ 
while $A''_i$ and $B''_i$ denote those vertices, 
that either belonged to $V_0,$ or changed cluster during the balancing procedure. Recall also, that $V''=(\cup_iA''_i) \cup (\cup_iB''_i)$ is the set of the {\it irregular} vertices. Let us fix $i\in [K]$ for the rest of this subsection, and without loss 
of generality, we describe, how to cover $A''_i\cup B''_i.$

If $|A''_i|, |B''_i|\le \ep^2m_i$ (note, that $m_i=|\widehat{A}_i|=|\widehat{B}_i|$),
then  $G[\widehat{A}_i, \widehat{B}_i]$ is a $(6\ep, d/3-\ep)$-super-regular pair by Lemma~\ref{beilleszt} and Fact~\ref{slicing}, and we do not need any further preparations,
we can apply the Blow-up lemma.

If not, either $A''_i$ or $B''_i$ is large, then we cannot guarantee super-regularity. 
In this case we will cover almost all irregular vertices with a few components from $T-T_S,$ the total
number of vertices in these components will be at most $4(|A''_i|+|B''_i|).$ Hence, the vast majority of 
$\widehat{A}_i\cup \widehat{B}_i$ will remain vacant. 

We will use a simple greedy method for covering the irregular vertices, which is outlined below in Fact~\ref{moho}. 
Proving the correctness of it is straightforward, we leave it for the reader. 

\begin{fact}\label{moho}
Let $\Gamma$ be a forest with levels sets $L_0, \ldots, L_k,$ and $J$ be a bipartite graph with parts $U_0$ and $U_1.$
Assume, that the partial embedding function $\phi$ have already been found for the first $t\le k-1$ level sets of $\Gamma$
such that $\phi(L_i)\subseteq U_{\langle i \rangle}$ for $0\le i\le t$ (here $\langle{i}\rangle= i \mod 2$).
Let $U'_{\langle t+1 \rangle}$ denote the vacant subset of $U_{\langle t+1 \rangle}.$ If $deg(u, U'_{\langle t+1 \rangle})\ge |L_{t+1}|$ for every $u\in \phi(L_t),$ then $\phi$ can be extended
to a partial embedding of the first $t+1$ levels of $\Gamma.$   
\end{fact}

As it was mentioned earlier, the embedding of the forests of $T-T_S$ is done in increasing order of their imbalances. 
Hence, after finishing the second phase, the imbalances of  unembedded forests will be at least as large as the largest
imbalance of those forests, which have already been embedded. 

\smallskip
  
We will follow a general scheme. 
Let us consider $A''_i$ 
(analogous method is used for $B''_i$): using Lemma~\ref{kimerito} below, we embed components of 
$T-T_S$ one by one into $G[\widehat{A}_i, \widehat{B}_i],$ such that at least about 25\% of the vertices in the components are mapped onto
vertices of $A''_i,$ and at most about 75\% are mapped into $G[A'_i, B'_i].$ Only the first three levels of a component could be scattered around in $G,$ since for these we use Lemma~\ref{connect}.
Therefore, after finishing the third phase, at most about $3|A''_i|$ and $3|B''_i|$ vertices will be covered in $A'_i$ and 
in $B'_i,$ respectively. Hence, what is left will still be a $2\ep$-regular pair. 

The key lemma for the third phase is Lemma~\ref{kimerito} below. Before stating it, we need preparations. 

Using random coin flipping, we choose random subsets
$A^0_R\subset A'_i$ and $B^0_R\subset B'_i.$ Applying Chernoff's inequality, with high probability we have 
$(1-\ep)|A'_i|/2\le |A^0_R| \le (1+\ep)|A'_i|/2,$ and $(1-\ep)|B'_i|/2\le |B^0_R| \le (1+\ep)|B'_i|/2,$ which
already implies, using Fact~\ref{slicing}, that the subgraph $G[A^0_R, B^0_R]$ is $3\ep$-regular and has density $d_i\pm 2\ep.$ 

We also have that $deg(v, A'_i) (1/2-\ep)\le deg(v, A^0_R)\le deg(v, A'_i)(1/2+\ep)$ for every $v\in \widehat{B}_i$
and $deg(v, B'_i) (1/2-\ep)\le deg(v, B^0_R)\le deg(v, B'_i)(1/2+\ep)$ for every $v\in \widehat{A}_i$ with high probability.
Recall, that if $v\in A''_i,$ then $deg(v, B'_i)\ge \gamma^3|B_i|/3,$ and similarly, if $w\in B''_i,$ then 
$deg(w, A'_i)\ge \gamma^3|A'_i|/3.$
Hence, with high probability, every $v\in A''_i$ will have at least $\gamma^3|B^0_R|/4$ neighbors in $B^0_R,$
and $deg(w, A^0_R)\ge \gamma^3|A^0_R|/4$ for every $w\in B''_i.$ 

For covering $A''_i$ we will embed components in $T-T_S$ into the subgraph $G[A^0_R\cup A''_i, B^0_R],$ 
except at most a constant number of vertices. 
Similarly, we use the subgraph
$G[A^0_R, B^0_R\cup B''_i]$ and a bounded number of vertices from the rest of $V(G)$ for covering vertices of $B''_i.$
 
As we cover more and more vertices, the unoccupied subsets of $A^0_R, B^0_R$ and $A''_i, B''_i$ shrink. 
Let us denote the unoccupied subset of $A^0_R$ by $A_R,$ the unoccupied subset of $B^0_R$ by $B_R,$ and the 
unoccupied subsets of $A''_i$ and
$B''_i$ by $U_A$ and $U_B,$ respectively.
Our goal is to prove that if $|A_R|$ and $|B_R|$ are sufficiently large, then we can embed a new component $F\subset T-T_S$
such that at least $|F|/4-2D^2$ vertices are covered in $U_A$, or in $U_B.$

By Proposition~\ref{Tdecomp}, a component $F$ of $T-T_S$ may have either one or two neighbors in $T_S.$ Since certain technical difficulties arise, when $F$ has two neighbors in $V(T_S),$ 
the proof of Lemma~\ref{kimerito} below is divided into two cases. We first assume that between $F$ and $T_S$ there
is precisely one edge. Then the case of two edges will be easy to deduce from the one edge case.  

\begin{lemma}\label{kimerito} 
Let $F$ be an unembedded component in $T-T_S$ with $|F|\le \eta n \le 5\ep^4 m_i,$ where $m_i= |\widehat{A}_i|=|\widehat{B}_i|.$ 
If $F$ has one neighbor, $s\in V(T_S),$ then let $xs\in E(T)$ denote the edge which connects $F$ to $T_S.$  If $F$ has two 
neighbors in the skeleton, $s$ and $s',$ then denote the two connecting edges by $xs$ and $x's'.$    
Assume that the partial embedding function $\Phi$ has already been constructed for $V(T_S)\cup V',$ where 
$V'\subset V(T)-V(T_S)$ with
$|V'|\le \gamma n/20.$ 
Assume further that $|A_R|\ge |A^0_R|-5|A''_i|,$ $|B_R|\ge |B^0_R|-5|B''_i|,$ and $|U_A| \ge \ep^3m_i.$ 
Then we can extend $\Phi$ for $V(F)$ such that 
\begin{itemize}
\item $\Phi(x)\Phi(s)\in E(G),$ and if $x's'$ exists, $\Phi(x')\Phi(s')\in E(G),$

\item $|\Phi(V(F))\cap (U_A\cup A_R\cup B_R)|\ge |F|-2D^3,$ 

\item and we cover at least $|F|/4-2D^2$ vertices of $U_A.$ 
\end{itemize}
Analogous statement holds for $U_B$ in place of $U_A,$ if $|U_B|\ge \ep^3m_i.$
\end{lemma}

\noindent {\bf Proof:} Without loss of generality we will consider the case of covering a subset of $U_A.$ 
Since $A_R$ and $B_R$ are large subsets of $A'_i$ and $B'_i,$ respectively, the following is immediate by Proposition~\ref{Gdecomp} and Fact~\ref{slicing}:

\begin{fact}\label{kvazirnd}
The subgraph $G[A_R, B_R]$ is $5\ep$-regular with density $d_i\pm 2\ep.$
\end{fact}

Let $\widetilde{B}$ denote those vertices of $B_R$ that have at least $\gamma^3|U_A|/10>\gamma^3\ep^3m_i/10>|F|$ neighbors in $U_A.$
We have the following lower bound for $|\widetilde{B}|,$ which always holds, if $U_A$ is sufficiently large. Note, that we stop the covering procedure if $|U_A|$ becomes smaller that 
$\ep^3m_i.$

\begin{claim}\label{Bhullam}
If $|U_A|\ge \ep^3m_i,$ then $|\widetilde{B}|\ge \gamma^3|B_R|/10.$
\end{claim}

\noindent {\bf Proof:} (of the claim) The number of edges between $U_A$ and $B^0_R$ is at least 
$\gamma^3|U_A| |B^0_R|/4.$ 
Since $|B_R|\ge |B^0_R|-5|B''_i|$ by our assumption and $|B''_i|\le \gamma^{4}|\widehat{B}_i|$ by Proposition~\ref{Gdecomp}, 
$$e(U_A, B_R)\ge \gamma^3|U_A| |B_R|/5.$$ Denote 
$\alpha\in (0,1)$ the proportion of vertices in $B_R$ which have less than $\gamma^3|U_A|/10$ neighbors in $U_A.$
Then we have the following inequality:
$$\alpha \gamma^3|B_R| |U_A|/10 + (1-\alpha) |B_R| |U_A| \ge \gamma^3 |B_R| |U_A|/5.$$ 
Simple computation gives, that $\alpha \le 1-\gamma^3/10,$ hence, $1-\alpha \ge \gamma^3/10,$ as desired.
\qedf

\medskip

Fact~\ref{kvazirnd} and Claim~\ref{Bhullam} imply the following. 

\begin{fact}\label{ARbol}
There are at least $(1-3\ep)|A_R|$ vertices in $A_R$ which all have at least 
$$(d_i-2\ep)|\widetilde{B}|\ge d_i \gamma^3|B_R|/15>5\ep |A_R|$$ neighbors in 
$\widetilde{B}.$ 
\end{fact}

\medskip

When we embed a component $F$ in order to cover a relatively large portion of $U_A,$ we follow a ``zigzag'' scheme,
possibly except for the first few levels of $F$: from $U_A$ we find neighbors in $B_R,$ then continue to $A_R,$ from there
to $\widetilde{B},$ and then we arrive back to $U_A.$ This scheme is being repeated until $F$ is embedded.   

By Proposition~\ref{Tdecomp}, a component $F\subset T-T_S$ can have one or two neighbors in $T_S.$ First we prove the lemma for the case when $F$ has precisely one vertex $x,$ which has a neighbor 
$s\in V(T_S).$ 

\medskip

\noindent {\bf Case I:} {\it $F$ has precisely one neighbor in $T_S$}

\smallskip
 
For $t\in \{0, 1, 2, 3\}$ we let $$\cL(t)=\bigcup_{j\ge 0} L_{t+4j}(F).$$
Set $\tau$ to be an index for which $|\cL(\tau)|\ge |\cL(t)|$ for every $t\in \{0, 1, 2, 3\}.$ Hence, $|\cL(\tau)|\ge |F|/4.$
The embedding of the first couple of levels of $F$ depends on the value of $\tau.$ We give the details for different values of $\tau$
as an itemized list below. As soon as we reached $U_A$ at some
level of $F,$ the rest of the embedding is the same for every value of $\tau,$ following the above mentioned zigzag
scheme, while applying the method of Fact~\ref{moho}.

\begin{itemize}

\item [The first case: $\tau=0.$] 

Let $H\subset A_R$ be a set with $|H|=3D^3/\gamma$ such that 
every vertex of $H$ has at least $D^3$ neighbors in $\widetilde{B}.$ We can choose $H$ by Fact~\ref{ARbol}.
Let $W=\Phi(V(T_S))\cup \Phi(V')\cup H.$ It is easy to see that $|W|\le \gamma n/10.$

Next we apply Lemma~\ref{connect} to connect $u=\Phi(s)$ with the set $H$ while avoiding $W.$ 
Note, that $L_0(F)=\{x\},$ and
$L_1(F)=N_T(x)-\{s\}.$ We embed the first three levels
of $F$ such that $\Phi(s)\Phi(x)\in E(G),$ $\Phi(L_0(F)\cup L_1(F))\subset V(G)-W$ and 
$\Phi(L_2(F))\subset H.$ 
Then choose $|L_2(F)|$ vertices arbitrarily from $H,$ and continue the embedding into $\widetilde{B}$ -- this is doable, since 
$|L_3(F)|<D^3.$ 
Using that every vertex of $\widetilde{B}$ has at least $\gamma^3|U_A|/10> |F|\ge |L_4(F)|$ neighbors in $U_A,$ we can embed the first
five levels of $F$ so that $\Phi(L_4(F))\subset U_A.$ In this case $\Phi(L_0)\not\subseteq U_A.$ 

\item [The second case: $\tau=1.$] 

As before, we apply Lemma~\ref{connect} with $u=\Phi(s)$ and $W,$ but 
this time $H\subset B_r$:  we choose $H$ so that $deg(h, A_R)\ge d_i|A_R|/2$ for every $h\in H.$ By $5\ep$-regularity
of $G[A_R, B_R]$ this is possible. In addition, we have $|H|=3D^3/\gamma.$ 

For $\Phi(L_3(F))$ we choose vertices from $N(H)\cap A_R,$ which all have at least $|L_4(F)|$ neighbors in 
$\widetilde{B}.$ Since
$deg(h, A_R)\ge d_i|A_R|/2> |L_3(F)|+5\ep|A_R|,$ we have enough room. Then, as in the previous case, from 
$\widetilde{B}$ 
we continue the embedding to $U_A.$ This time $\Phi(L_5)\subset U_A,$ but $\Phi(L_1)\not\subseteq U_A.$ 

\item [The third case: $\tau=2.$]

This time we may choose $H$ from $U_A$ arbitrarily, since every irregular vertex has many neighbors,
hence, we have $\Phi(L_2)\subseteq U_A.$ 

\item[The fourth case: $\tau=3.$]

In this case we let $H\subset \widetilde{B}$ with $|H|=3D^3/\gamma,$ set $W$ as before, apply Lemma~\ref{connect}, and then continue to embed $L_3(F)$ into $U_A,$ hence, $\Phi(L_3)\subseteq U_A.$ 

\end{itemize}

After reaching $U_A,$ we use the zigzag scheme with the method of Fact~\ref{moho}, no matter what the value of 
$\tau$ is. We proceed level by level. From $U_A$ we can always continue the embedding
to $B_R,$ then to such vertices of $A_R$ which all have at least $|F|$ neighbors in $\widetilde{B},$ then
from $\widetilde{B}$ we reach $U_A$ again. Fact~\ref{ARbol} is crucial, it states that almost all vertices of $A_R$ have
at least $5\ep |A_R|>|F|$ neighbors in $\widetilde{B}.$ Using $5\ep$-regularity of $G[A_R, B_R],$ it is easy to reach such vertices 
of $A_R$ from almost all of $B_R.$ Hence, we never get stuck, and following the zigzag scheme we can embed every fourth level into $U_A.$ One can also check easily that $|\Phi(\cL(\tau))\cap U_A|\ge |F|/4-(D-1),$ and we ``lose'' $D-1$
only when $\tau=1$ (this loss is 1, if $\tau=0,$ and 0, if $\tau=2, 3$). 
This finishes the proof for the case when $F$ has precisely one neighbor in $T_S.$

\medskip

\noindent {\bf Case II:} {\it $F$ has precisely two neighbors in $T_S$}

\smallskip

Now we assume, that $F$ has two neighbors in the skeleton. By Proposition~\ref{Tdecomp} we know that there exist
$s, s'\in V(T_S)$ such that in $T'$ these vertices are connected by a long maximal line $\ell.$ All inner vertices of $\ell$ belong
to $F.$ Say, that $\ell= s e_1 x_1 e_2 x_2 e_3\ldots e_t x_t e_{t+1} s',$ here $|\ell|=t+1,$ and, by definition, $|\ell|\ge 10.$

We split $F$ into several subtrees, and embed them one by one. The first subtree, $F_1$ is the component of $x_1$ in $F-\{x_{t-1}, x_t\}.$ 
This component has one neighbor, $s\in V(T_S),$ hence, we can embed it with the method of {\bf Case I}. This implies that
the loss is at most $D-1$ at this point.

Next we need a simpler version Lemma~\ref{connect}, which is a direct consequence of the $\gamma$-non-extremality of $G.$
We leave the proof for the reader. 

\begin{claim}\label{ossze}
Let $u, v\in V(G)$ be two distinct vertices, and $W\subset V(G)$ be a set with $|W|\le \gamma n/10.$ Then $G$ has a $u-v$-path
of length three, which avoids $W.$
\end{claim}
  
Note that we have already found $\Phi(x_{t-2})$ when embedding $F_1.$ Using Claim~\ref{ossze} we find a length-3 path 
$uf_1w_1f_2w_2f_3v,$ ($f_i\in E(G)$) such that $u=\Phi(x_{t-2}),$ $v=\Phi(s'),$ and the images of $x_{t-1}$ and $x_t$ are $w_1$ and $w_2,$
respectively. 

Let $F_2$ denote the subtree which we obtain as the component containing $x_{t-1}$ in $F-\{x_1, \ldots, x_{t-2}, x_t\}.$
Since $x_{t-1}$ has two neighbors in $\ell,$ it may have up to $D-2$ neighbors, $y_1, \ldots, y_q.$ Every $y_i$ is the root
of a component in $F_2-x_{t-1},$  which has precisely one neighbor, $x_{t-1}.$ For each of these components we can apply the method
of {\bf Case I} in order to embed them. This results in a loss of at most $(D-2)(D-1).$ 

Finally, let $F_3$ denote the subtree which we obtain as the component containing $x_{t}$ in $F-\{x_1, \ldots, x_{t-2}, x_{t-1}\}.$ 
Similarly to the case of $F_2,$ $x_t$ will have up to $D-2$ neighbors in $F_3,$ $z_1, \ldots, z_r,$ each being the root of
a component in $F_3-x_t.$ These components can be embedded with the method of  {\bf Case I}, Again, we will have a loss
of at most $(D-2)(D-1).$

We can use the bounds for $|F|/4 - |\Phi(F)\cap U_A|$ separately for each subtree. Since the number of subtrees in
{\bf Case II} is $2D-3,$ and we also have two vertices, $\Phi(x_{t-1})$ and $\Phi(x_{t}),$ which we may not map into
$U_A,$ altogether we have $$|\Phi(F) \cap U_A|\ge |F|/4-(2D-3)(D-1)+2\ge |F|/4-2D^2.$$
This finishes the proof of
the lemma.
\qedf

\section{The proof of Theorem~\ref{fa}}\label{alg}

We have all the tools needed for discussing the proof of Theorem~\ref{fa}. The embedding function $\Phi$ is determined in three phases.

\subsection{The first phase: preprocessing of $G$ and $T,$ and embedding the skeleton $T_S$}

The first phase consists of three parts. 

\begin{itemize}

\item [(i)] Preprocessing of $G$: Apply Theorem~\ref{pontfelbontas} and then Proposition~\ref{Gdecomp} for finding the $G[\widehat{A}_i, \widehat{B}_i]$ pairs for every $i\in [K].$ Set $m= d^{101/\ep^2}n/4,$ this is a lower bound for the cardinality of the smallest cluster in
the decomposition, that is, $m\le \min\{m_i: i\in [K]\}.$  
Let $\K_T= 20D^2n/(\ep^4 m)= 80D^2d^{-101/\ep^2}/\ep^4,$ a constant

\item[(ii)] Preprocessing of $T$: Apply Proposition~\ref{Tdecomp} with parameter $\eta=20D^2/\K_T\ll \ep$ for finding the decomposition of $T$ into the skeleton $T_S$ and the vertex-disjoint components $F_1, \ldots, F_k \subset T-T_S$ such that
$|F_i|\le \ep^4m$ for every $i\in [K],$ and $k\le \K_T.$ 

\item[(iii)] Embed $T_S$ into $G[\widehat{A}_1, \widehat{B}_1]$ greedily. This is easy, as $|T_S|\le 20Dn/(\ep^4m)=\K_T/D,$
and the minimum degree in $G[\widehat{A}_1, \widehat{B}_1]$ is much larger. With this we have determined
$\Phi(V(T_S)).$

 \end{itemize}

\subsection{The second phase: covering the irregular vertices}

Proposition~\ref{Gdecomp} can only guarantee that large portions of the cluster pairs are quasirandom, 
but they may contain a non-negligible number of irregular vertices. Below we discuss how to use Lemma~\ref{kimerito} for achieving this goal.

\medskip

Before using Lemma~\ref{kimerito} we form ``large chunks'' from the components of $T-T_S.$ These are forests, with
cardinality in the range of $\ep^4m$ and $2\ep^4m,$ 
except possibly one, which may contain more than $2\ep^4m,$ but less than $3\ep^4m$ vertices. 
Since we keep the vertices of short maximal lines in the skeleton, small components may appear in $T-T_S,$ if these have their neighbors in short maximal lines.
Hence, there could be some components  $F_i\subset T-T_S$ such that 
$|F_i|\ll \ep^4m.$ The following simple algorithm gives us the desired large chunks $CH_1, \ldots, CH_t$ from
the components $F_i\subset T-T_S,$ $i=1, \ldots, k.$

\begin{enumerate}

\item Organize the $F_i$ components into a list $\Lambda$ in size-increasing order: $\Lambda=F_1,\ldots, F_k,$ 
such that $|F_1|\le |F_2|\le \dots \le |F_k|.$

\item If $\sum_{F\in \Lambda} |F|\ge \ep^4m,$ then let $\Lambda'$ denote the sub-list containing the smallest prefix of $\Lambda$ such that total cardinality of the components in $\Lambda'$ is at least $\ep^4m.$ The new chunk is 
$CH=\bigcup_{F\in \Lambda'}F.$ Delete the elements of $\Lambda'$ from $\Lambda,$ and keep the notation 
$\Lambda$ for
what is left in the list. 

\item If $\sum_{F\in \Lambda} |F|< \ep^4m,$ then add all the components in $\mathcal{L}$ to the
most recently formed chunk, and stop.

\item If $\Lambda$ is empty, stop, otherwise continue with Step 2.
 
\end{enumerate}

\medskip

The following fact is immediate.

\smallskip

\begin{fact}\label{csank1}
The number $t$ of chunks is at most $\K_T.$ Every chunk $CH_i$ ($i=1, \ldots, t$) has cardinality between $\ep^4m$ and $3\ep^4m.$ Each chunk is a forest, 
containing components of $T-T_S.$
\end{fact}

\medskip

Recall, that the imbalance was defined not only for trees, but for forests as well. 
We will embed the chunks in {\it increasing order} of their imbalances.
If necessary, we change the indices of the chunks, and from now on we assume that $Imb(CH_i)\le Imb(CH_{i+1})$
for every $1\le i\le t-1.$

\medskip

By Lemma~\ref{beilleszt}, if less than a proportion of $\ep^2$ vertices is inserted to the clusters of an $\ep$-regular pair, 
the pair will remain quasirandom with a slightly worse parameter: it will become $3\ep$-regular. Hence, only that case is interesting for us when
there are relatively many irregular vertices in a pair. We cover the irregular vertices in the $G[\widehat{A}_i, \widehat{B}_i]$ pairs for every $i\in [K].$ Fix an arbitrary $i\in [K],$
and assume, without loss of generality, that $|A''_i|\ge \ep^3 m_i.$ Call a chunk {\it available}
if it have not been mapped yet. Let $CH$ denote a chunk having minimum imbalance among the available ones. 
Assume, that it contains $j$ components, $F_{i_1}, \ldots, F_{i_j}.$ Apply Lemma~\ref{cover} for all the components
in $CH,$ this covers at least $|CH|/4-2jD^2$ irregular vertices of $A''_i.$ Then make the chunk {\it unavailable}.

Note, that since the number of components in $CH$
is at most $\K_T,$ the number of vacant irregular vertices have been decreased by 
more than $|CH|/5\ge \ep^4m/5.$ 
If there are still more than $\ep^3m$ unoccupied irregular vertices in $A''_i,$ repeat the above procedure with
an available chunk having minimum imbalance. When we are done with $\widehat{A}_i,$ we repeat the procedure for 
$\widehat{B}_i.$

Since in the beginning we had that $|A''_i|, |B''_i|\le \gamma^4 m_i,$ we cover less than $5\gamma^4 m_i$ vertices
in $\widehat{A}_i$ and in $\widehat{B}_i.$ Moreover, since 
$\sum_{i\in [K]}|\widehat{A}_i\cup \widehat{B}_i|)=n,$ the total number of vertices in the chunks we used for the 
covering is less than $10\gamma^4 n.$ This, together with Fact~\ref{csank1} implies the following.

\begin{fact}\label{csank2}
The total number of chunks used for covering the irregular vertices in $\widehat{A}_i$ ($i\in [K]$) is less than
$5 \gamma^4 m_i/(\ep^4m),$ and their total imbalance is less than $5\gamma^4 m_i.$
Similar holds for $\widehat{B}_i.$
\end{fact}

For every $i\in [K]$ there are at most $\ep^3 m_i$ vacant irregulars vertices which remained in $\widehat{A}_i$
and $\widehat{B}_i,$ respectively, and the total number of vacant vertices in both clusters of the $i$th pair is more than $(1-5\gamma^4) m_i.$ We introduce the notation 
$\widetilde{A}_i\subset \widehat{A}_i$ and $\widetilde{B}_i\subset \widehat{B}_i$ for the vacant parts.

\begin{obs}\label{szupregi}
By Proposition~\ref{Gdecomp}, Lemma~\ref{beilleszt} and Fact~\ref{slicing}, the $G[\widetilde{A}_i, \widetilde{B}_i]$ pairs are $3\ep$-regular
for every $i\in [K].$  
Recall, that in Lemma~\ref{kimerito} we used a random subset for the covering, obtained by random coin flipping. 
Hence, the degrees in a pair are still large enough with high probability, every vertex is adjacent to at least a proportion of $d/3-2\ep>d/4$
vertices in the opposite part. That is, we can work with $(6\ep, d/4)$-super-regular pairs from now on.
Although these may not be balanced, but not far from it: $||\widetilde{A}_i|-|\widetilde{B}_i||\le 10\gamma^4|\widehat{A}_i|$ for every $i.$
\end{obs}



\subsection{The third phase}

By Observation~\ref{szupregi} every pair is $(3\ep, d/4)$-super-regular.
This super-regularity is what we need for finishing the embedding of $T$ by the help of the Blow-up lemma~\cite{KSSzBl2}.

\medskip

\begin{lemma} [Blow-up Lemma]
Given a graph $R$ of order $r$ and positive
parameters $\delta, D,$ there exists a positive $\ep = \ep(\delta, D, r)$ such that the following holds. Let
$n_1, n_2, \ldots, n_r$ be arbitrary positive integers and let us replace the vertices $v_1, v_2, \ldots, v_r$ of $R$
with pairwise disjoint sets $V_1, V_2, \ldots, V_r$ of sizes $n_1, n_2, \ldots, n_r$ (blowing up). We construct two
graphs on the same vertex set $V = \cup V_i.$ The first graph ${\mathbf R}$ is obtained by replacing each
edge $v_iv_j$ of $R$ with the complete bipartite graph between the corresponding vertex sets
$V_i$ and $V_j.$ A sparser graph $G$ is constructed by replacing each edge $v_iv_j$ with an $(\ep, \delta)$-super-regular 
pair between $V_i$ and $V_j.$ If a graph $H$ with $\Delta(H)\le D$ is embeddable into ${\mathbf R}$
then it is already embeddable into $G.$

Moreover, the following strengthening also holds. Given $c>0,$ there are positive numbers $\ep=\ep(\delta, D, r, c)$
and $\alpha=\alpha(\delta, D, r, c)$ such that $H$ is embeddable into $G,$ if, for
every $i,$ there are certain vertices $x$ to be embedded into $V_i$ whose images are a priori
restricted to certain sets $C_x \subset V_i$ provided that
\begin{itemize}
\item[(i)] each $C_x$ within a $V_i$ is of size at least $c|V_i|,$ and
\item[(ii)] the number of such restrictions within a $V_i$ is not more than $\alpha|V_i|.$
\end{itemize}
\end{lemma}

\medskip

In our case the graph $R$ is a matching on $K$ edges, and $\mathbf{R}$ is the set of super-regular pairs $G[\widetilde{A}_i, \widetilde{B}_i]$ for $i\in [K].$ The graph $H$ to be embedded is the union of the available chunks.

\medskip

Below we sketch the final steps of the embedding as an itemized list.

\begin{itemize}

\item [Step 1.] First we will assign the chunks to the clusters in the pairs such that
for every cluster the difference of the number of assigned tree vertices and the size of the cluster is at most $3\ep^4m.$ 

\item [Step 2.] Use Lemma~\ref{connect} for finding the connections between the clusters of assigned chunks and the skeleton.

\item [Step 3.] Use Lemma~\ref{reloc}, the relocation lemma, in order to achieve that the number of tree vertices assigned to a cluster is equal to the size of that cluster. 

\item [Step 4.] Apply the Blow-up lemma.

\end{itemize} 

\medskip

We analyze the above steps as follows.

\subsubsection{Step 1.}

For any subset $\cC$ of a cluster let $\cT(\cC)$ denote the set of tree vertices which are assigned to $\cC.$ 
Right after the second phase
$\cT(\widetilde{A}_i) = \cT(\widetilde{B}_i)=\emptyset$ for every $i\in [K],$ since $\widetilde{A}_i, \widetilde{B}_i$
denote the vacant subclusters of $\widehat{A}_i$ and $\widehat{B}_i,$ respectively. 
By Observation~\ref{szupregi}, the total imbalance of chunks embedded into the $i$th pair 
$G[\widetilde{A}_i|, \widetilde{B}_i|]$ is
 $||\widetilde{A}_i|-|\widetilde{B}_i||\le 10\gamma^4 m_i.$ 

Recall, that at this point every {\it available} chunk has imbalance
at least as large as the largest imbalance of chunks in the second phase. We use a simple method for assigning
nearly the same number of tree vertices to clusters as the cluster size. For every $i\in [K]$ we repeat the following:
take an available chunk with smallest imbalance, and assign its parts to clusters of the $i$th pair so that  if 
$|\widetilde{A}_i|-|\cT(\widetilde{A}_i)| \le |\widetilde{B}_i|-|\cT(\widetilde{B}_i)|,$ then the larger 
part is assigned to $\widetilde{A}_i$
and the smaller is assigned to $\widetilde{B}_i.$ Otherwise assign the larger
part to $\widetilde{B}_i$ and the smaller part to $\widetilde{A}_i.$ After every assignment update the sets $\cT(\widetilde{A}_i)$
and $\cT(\widetilde{B}_i).$
We stop assigning chunks to the $i$th pair, if both of its $\cT$-sets become at least as large as its cluster,
that is, when $|\cT(\widetilde{A}_i)|\ge |\widetilde{A}_i|$ and $|\cT(\widetilde{B}_i)|\ge |\widetilde{B}_i|$. If $i<K,$
we continue the algorithm with the $(i+1)$st pair. 

The statement below follows from the fact that the imbalances of the chunks in the third phase are not smaller than the
imbalances of chunks in the second phase. We leave the proof for the reader. 

\begin{lemma}\label{csankok}
(1) Assume, that an $r$ element subset of chunks, denoted by $CH,$ was used for covering the irregular vertices of the $i$th pair for 
$i\in [K]$ in the second phase. Then we can assign a subset of chunks $CH'$ with at most  $3r$ chunks such that $|Imb(CH)-Imb(CH')|\le 3\ep^4 m.$ This follows from the fact that every chunk has size between $\ep^4m$ and $3\ep^4m.$

(2) After assigning $CH',$ at every iteration step of the above assigning algorithm we have that $$\left||\cT(\widetilde{A}_i)| -|\cT(\widetilde{B}_i)|\right|\le 3\ep^4m,$$ which is
an upper bound for the size of chunks. 

(3) Every chunk will be assigned by the algorithm to some pair, no chunk will
remain unassigned. This follows
from the fact that the total number of vertices in available chunks is equal to the number of vacant vertices in $V(G).$  

(4) By the stopping rule, we have that 
$$\left||\widetilde{A}_i|-|\cT(\widetilde{A}_i)|\right|\le 3\ep^4m \ and \  \left||\widetilde{B}_i|-|\cT(\widetilde{B}_i)|\right|\le 3\ep^4m$$ for every $i\in [K].$ 
\end{lemma}

\subsubsection{Step 2.}

Given an arbitrary pair $G[\widetilde{A}_i, \widetilde{B}_i]$ for $i\in [K],$ let $F$ be a component, which was assigned to this pair. We will use Lemma~\ref{connect} in order to connect $F$ to its neighbor or neighbors in
$T_S,$ the same way as it was described in Lemma~\ref{kimerito} for connecting those tree components to $T_S$ which
covered the irregular vertices. When we find one such connection to the skeleton, we fix the image of 
$|L_0(F)\cup L_1(F)\cup L_2(F)\cup L_3(F)|\le 2(1+D+D^2+D^3)$ vertices. Recall, that $\Phi(L_3(F))\subset \widetilde{A}_i \cup
\widetilde{B}_i,$ while we do not have full control over the images of $L_0(F)\cup L_1(F)\cup L_2(F),$ although we can
rule out a set $W\subset V$ with $|W|\le \gamma n/10.$  

Assume that $\Phi(x)\in \widetilde{A}_i$ for $x\in L_3(F).$ We prepare for the application of the Blow-up lemma. For every $y\in N_T(x)\cap L_4(F)$ we let 
$C_y=N_G(\Phi(x)\cap \widetilde{B}_i).$ If $\Phi(x)\in \widetilde{B}_i,$ then let $C_y=N_G(\Phi(x)\cap \widetilde{A}_i).$
In the first case we have $|C_y|\ge d|\widetilde{B}_i|/4,$ in the second case $|C_y|\ge d|\widetilde{A}_i|/4,$ here we used Observation~\ref{szupregi}. The $C_y$ set is the {\it restriction set}
for the image of the {\it restricted vertex} $y.$ Restriction sets of restricted vertices will be used in Step 4, when applying the Blow-up lemma.

Since altogether the number of tree components is at most $\K_T,$ in this step we fix the image of a total of at most
$2\K_T (1+D+D^2+D^3)<3D^3\K_T$ vertices. This implies, that the number of restricted vertices is less than 
$D\cdot 3D^3\K_T=3D^4\K_T,$ a constant.

In the beginning of Step 2 the set $W$ will contain those vertices which were covered during the second phase, 
and whenever we fix the image of a vertex in this step, we add it to $W.$ 
During phase 2 we covered at most $10\gamma^4n$ vertices, and in this step less than $3D^3\K_T$ vertices get covered. 
Hence, $|W|\le \gamma n/10$ at every point in time, so we never get stuck when finding the connection of tree components
with the skeleton.

\subsubsection{Step 3.}

Let $\cC\in \bigcup_{i\in [K]}(\widetilde{A}_i \cup \widetilde{B}_i)$ be an arbitrary cluster. We say, that it has a {\it deficiency}, if
$|\cC|<|\cT(\cC)|,$ and has a {\it surplus}, if $|\cC|>|\cT(\cC)|.$ We will use Lemma~\ref{reloc} for eliminating
deficiencies, and therefore, surpluses of clusters -- clearly, if there is a cluster with deficiency, then there must exist a
cluster with surplus and vice versa. 

Using Lemma~\ref{csankok}, the total deficiency (and therefore the total surplus) of clusters
is at most $K\cdot 6\ep^4m \le 6\ep^4n.$ This means that applying Lemma~\ref{reloc} at most $6\ep^4n$ times we may
achieve that no cluster in the decomposition has either deficiency or surplus. 
It is crucial to do the relocations so that super-regularity is not destroyed in any cluster, and the constant number 
of restriction sets must remain large enough after the relocations.  
A very similar task was solved in Section~\ref{balanceproc} when distributing the irregular vertices, so we will only
sketch the procedure.

Lemma~\ref{reloc} guarantees the existence of at least $\gamma^6n/30$ disjoint relocation paths for any vertex $v$ and
any cluster $\cC,$ where $\cC$ is a cluster obtained by Theorem~\ref{pontfelbontas}. Fix $v$ and $\cC.$ 
Some of the $v-\cC$ relocation paths were possibly destroyed: $(I)$ in the first phase, in the preprocessing of $G,$ we
distributed the irregular vertices, $(II)$ still in the first phase the skeleton was embedded, $(III)$ in the second phase
we covered the irregular vertices, and finally, $(IV)$ in Step 2 of the third phase we fixed the image of the first three levels of
components. Note, that in $(II)$ and $(IV)$ only a constant number of vertices are effected, so we focus our attention
to the other two cases.

Recall, that for distributing the irregular vertices we restricted ourselves for using only a random subcluster for every 
cluster. That is, almost half of every cluster remained intact. Similarly, when covering the irregular vertices, random
subclusters were used, moreover, these random choices were independent. Assume, that $(v, u, w)$ is a relocating
triple. The probability that neither $u,$ nor $w$ were included in subclusters when relocating the irregular vertices is $1/4,$
and similarly, the probability that neither $u,$ nor $w$ were included in subclusters used for covering the
irregular vertices is $1/4.$ Hence, with probability $1/16$ both $u$ and $w$ belong to intact subclusters.
By Chernoff's bound, with high probability the number of $v-\cC$ relocation paths is at least $\gamma^6n/30 \cdot \frac{1}{20}=\gamma^6n/600.$

Assume, that $\cC_1$ and $\cC_2$ are clusters such that $\cC_1$ has a surplus and $\cC_2$ has a deficiency. Then we pick an arbitrary $v\in \cC_1$ and 
among the available ones, randomly choose a $v-\cC_2$ relocating path. This is done at most $6\ep^4n$ times, hence, 
there are always more than
$\gamma^6n/1000$ relocation paths to choose from, even for eliminating the last surplus.
Using martingales and Azuma's inequality as in Section~\ref{balanceproc}, we get that with high probability
every cluster $\cC$ will participate in less than $\ep^3|\cC|$ relocations with high probability.
 
Hence, for every cluster $\cC$ and every $v\in \cC$ the degrees may change a little, by less than $\ep^3|\cC|.$ 
Using Lemma~\ref{beilleszt} we also have that the pairs remain $(9\ep, d/4-\ep)$-
super-regular, as they were $(3\ep, d/4)$-super-regular before. Finally, since the restriction sets in 
cluster $\cC$ had at least $d|\cC|/4$ vertices, after this step every restriction set has more than $d|\cC|/5$ vertices.

\subsubsection{Step 4.}

Since every requirement of the Blow-up lemma is satisfied, we can apply it, and embed the majority of $T.$ 
This finishes the embedding, and thereby the proof of Theorem~\ref{fa}.\qedf

\end{document}